\documentclass[12pt]{amsart}
\usepackage{amssymb}
\usepackage{amsmath,amscd}
\usepackage{amsfonts,latexsym,verbatim,amscd,mathrsfs,color,array}
\usepackage[all,cmtip]{xy}
\usepackage{mathrsfs}
\usepackage{multirow}
\usepackage{graphicx}
\usepackage{amsfonts, amsthm}
\usepackage{bbm}
\usepackage[hmargin=1in,vmargin=1in]{geometry}
\usepackage{mathdots}
\usepackage{stmaryrd}
\usepackage{MnSymbol}
\usepackage{bbm}
\usepackage[hidelinks]{hyperref}

\allowdisplaybreaks

\def\XXint#1#2#3{{\setbox0=\hbox{$#1{#2#3}{\int}$ }
		\vcenter{\hbox{$#2#3$ }}\kern-.6\wd0}}

\renewcommand{\Re}{\operatorname{Re}}

\DeclareMathSymbol{\intprod}{\mathbin}{MnSyC}{'270}
\newcommand{\LB}{\left[}
\newcommand{\RB}{\right]}
\newcommand{\LA}{\left\langle}
\newcommand{\RA}{\right\rangle}

\newcommand{\N}{{\mathbb N}}
\newcommand{\C}{{\mathbb C}}
\renewcommand{\P}{{\mathbb P}}

\newcommand{\R}{{\mathbb R}}

\newcommand{\eps}{\varepsilon}

\newcommand{\supp}{{\mathrm{supp} \, }}
\renewcommand{\div}{{\mathrm{div} }}

\newcommand{\dbar}{{\bar{\partial}}}
\newcommand{\T}{{\mathcal{T}}}

\newcommand{\dist}{\mathrm{dist}}

\newcommand{\pd}{{\partial}}

\newcommand{\inj}{{\mathrm{inj} }}
\newtheorem{thm}{Theorem}[section]
\newtheorem{lemma}[thm]{Lemma}
\newtheorem*{lemma*}{Lemma}
\newtheorem{prop}[thm]{Proposition}

\newtheorem{cor}[thm]{Corollary}

\newtheorem*{conj*}{Conjecture}

   \newtheoremstyle{others}
     {3pt}
     {2pt}
     {}
     {}
     {\bf}
     {.}
     {.5em}
     {}

\theoremstyle{others}
\newtheorem{rmk}[thm]{Remark}
\newtheorem*{rmk*}{Remark}
\newtheorem{defn}[thm]{Definition}

\numberwithin{equation}{section}

 \DeclareMathOperator{\tr}{tr}
 
 \newcommand{\norm}[1]{\Vert#1\Vert}
 
 \def\<{\left\langle} \def\>{\right\rangle}
 \def\({\left(} \def\){\right)}
 
 \newcommand{\be}{\beta}
 \newcommand{\ld}{\lambda}
 
 \newcommand{\de}{\delta}
 \newcommand{\De}{\Delta}
 \newcommand{\ep}{\varepsilon}
 \newcommand{\Si}{\Sigma}
 \newcommand{\si}{\sigma}
 \newcommand{\om}{\omega}
 \newcommand{\Om}{\Omega}
 \newcommand{\ga}{\gamma}
 \newcommand{\Ga}{\Gamma}
 \newcommand{\ka}{\kappa}
 \renewcommand{\th}{\theta}

 \renewcommand{\b}{\bar}
 \renewcommand{\o}{\overline}
 \newcommand{\p}{\partial}
 \newcommand{\n}{\nabla}
 \newcommand{\bp}{\b{\p}}

\newcommand{\E}{\mathcal{E}}

\begin{document}

\title{Harmonic map flow for almost-holomorphic maps}
\author{Chong Song}
\address{Xiamen University, Fujian, China}
\email{songchong@xmu.edu.cn}
\author{Alex Waldron}
\address{University of Wisconsin, Madison}
\email{waldron@math.wisc.edu}

\begin{abstract}
Let $\Si$ be a compact oriented surface and $N$ 
a compact K\"ahler manifold with nonnegative holomorphic bisectional curvature. For a solution of 
harmonic map flow starting from an almost-holomorphic map $\Sigma \to N$ (in the energy sense), the limit at each singular time extends continuously over the bubble points and no necks appear.
\end{abstract}

\maketitle


\thispagestyle{empty}

\tableofcontents

\section{Introduction}

\subsection{Background}

Let $(M, g)$ and $(N, h)$ be compact Riemannian manifolds. For any $C^1$ map $u: M \to N,$ we may define the \emph{Dirichlet energy}
\[ E(u)=\frac12\int_M|du|^2 dV_g.\]
Critical points of $E(u)$ are referred to as \emph{harmonic maps} and play a fundamental role in geometric analysis.


\emph{Harmonic map flow} is the downward gradient flow of the Dirichlet functional:
\begin{equation}\label{e:hm}
  \frac{\p u}{\p t} =\tr_g\n du.
\end{equation}
In 1964, Eells and Sampson~\cite{EellsSampson1964} introduced the evolution equation (\ref{e:hm}) and proved their foundational theorem: if $N$ has nonpositive sectional curvature, then 
harmonic map flow smoothly deforms any initial map to a harmonic map in the same homotopy class. 

The situation becomes more challenging when the target manifold is allowed to have positive sectional curvatures.
We shall focus on the case of a two-dimensional domain, 
where the Dirichlet functional is conformally invariant and enjoys a rich variational theory \cite{SacksUhlenbeck1981}. 
In 1985, 
Struwe~\cite{Struwe1985} constructed a global weak solution of (\ref{e:hm}) starting from an arbitrary $L^2_1$ initial map. 
The Struwe solution is regular away from a finite set of spacetime points where the Dirichlet energy may possibly concentrate, forming a singularity. 
In 1992, 
Chang, Ding, and Ye \cite{CDY1992} proved 
that these finite-time singularities
must occur in the scenario of rotationally symmetric maps between 2-spheres; they are an inevitable feature of the theory. 

During the 1990s, a detailed picture of singularity formation in 2D harmonic map flow emerged from the work of several authors. 
Given a singular time $0 < T< \infty,$ 
the limit
$$u(T) = \lim_{t\nearrow T} u(t),$$
referred to as the \emph{body map}, exists weakly in $L^2_1$ and smoothly away from the singular set $\{x_i\}_{i = 1}^K$ (this already follows from \cite{Struwe1985}). 
By rescaling around a well-chosen sequence of spacetime points approaching 
$(x_i, T),$ one obtains 
a \emph{bubble tree}\footnote{For an introduction to the bubble-tree concept, see Parker \cite{Parker2003}.}
consisting of finitely many harmonic maps
$$\phi_{i,j} : S^2\to N,$$ 
where $j = 1, \ldots, L_i.$
These satisfy the following \emph{energy identity} \cite{Qing1995, DingTian1995, Wang1996}: 
\begin{equation}\label{alltogetherenergyidentity}
\lim_{t \nearrow T}E(u(t))=E(u(T))+\sum_{i = 1}^K \sum_{j = 1}^{L_i} E(\phi_{i,j}).
\end{equation}
Moreover, \emph{necks} cannot form between the bubbles; i.e., for each $i,$ the subset
$$\cup_j \phi_{i,j}(S^2) \subset N$$ 
must be connected \cite{QingTian1997, LinWang1998}.

For some time, it was believed that the body map $u(T)$ should always extend continuously across the singular set \cite{LinWang1998,Qing2003}---that is,
until 2004, when Topping \cite{Topping2004-mathz} was able to construct an example 
in which $u(T)$ has an essential singularity. This may have caused interest in the subject to decline in the following years. 

However, Topping's counterexample depends on the construction of a pathological metric on the target manifold. 
In the same paper \cite{Topping2004-mathz}, he conjectured that when the metric on $N$ is sufficiently well behaved (specifically, real-analytic), 
$u(T)$ will have only removable discontinuities.
The conjecture is usually taken to include the statement that necks cannot form between the bubbles and the body map, which is currently known only at infinite time. 
Topping \cite{Topping2004-cvpde} proved that both properties follow from H\"older continuity of the Dirichlet energy $E(u(t))$ with respect to $t < T,$ although this assumption is difficult to verify in practice. 

\subsection{Main results}

This paper establishes the continuity and no-neck properties under 
a set of hypotheses familiar from the classical theory of harmonic maps \cite{ SiuYau1980, Wood1979} as well as later work on harmonic map flow \cite{Topping1997, LiuYang2010}. 
In the K\"ahler setting, the Dirichlet energy decomposes into holomorphic and anti-holomorphic parts:
\[ E(u)=\int |\p u|^2dV_g+\int |\bar{\p}u|^2dV_g =: E_{\p}(u)+E_{\bp}(u).\]
Topping's earlier work \cite{Topping1997} on rigidity of (\ref{e:hm}) at infinite time was based on assuming that $u$ is almost (anti-)holomorphic, i.e., that either
$E_{\p}(u)$ or $E_{\bp}(u)$ is small. 
Liu and Yang \cite{LiuYang2010} generalized
Topping's rigidity theorem using a Bochner technique that also requires a positivity assumption on the curvature of $N.$ 
We make the same assumptions in our main theorem, which follows.


\begin{thm}\label{thm:main}
Given a compact K\"{a}hler manifold $(N,h)$ with nonnegative holomorphic bisectional curvature, 
there exists a constant $\delta_0 > 0$ as follows.

Let $(\Sigma, g)$ be a compact, oriented, Riemannian surface and $u:\Si \times[0,\infty)\to N$ a weak solution (in Struwe's sense) of harmonic map flow with either $E_{\bp}(u(0)) <\delta_0$ or $E_{\p}(u(0)) < \delta_0.$
For each singular time $T < \infty,$ the map $u(T)$ is $C^\mu$ for each $\mu < 1.$ 
Moreover, 
given any sequence of times $t_n \nearrow T$ or $\infty,$ 
the maps $u(t_n)$ sub-converge in the bubble-tree sense, satisfying the energy identity (\ref{alltogetherenergyidentity}) and without necks---i.e., for each $i,$ $\cup_j \phi_{i,j}(S^2)$ 
is connected and contains $\lim_{x \to x_i} u(x,T).$

Here, we may take $\delta_0 = \frac{c}{\sup_N |H_N|},$ where $H_N$ is the holomorphic sectional curvature of $N$ and $c$ is a universal constant.
\end{thm}




Note that by the proof of the generalized Frankel conjecture due to Mori \cite{Mori1979}, Siu-Yau \cite{SiuYau1980}, and Mok \cite{Mok1988}, our curvature assumption implies that $N$ is biholomorphic to a Hermitian symmetric space, and to $\C\P^n$ in the strictly positive case. However, the metric $h$ need not be symmetric: for example, any metric on the 2-sphere with nonnegative curvature satisfies the hypothesis. Meanwhile, even in the case $\Sigma = N = S^2_{round},$ our results are new. 

We also note that the bubble-tree statement in Theorem \ref{thm:main} allows for an arbitrary sequence $t_n \nearrow T \text{ or } \infty,$ by contrast with the work discussed above.
This is possible because our results follow from parabolic estimates rather than from an analysis of sequences of low-tension maps.

\subsection{Outline of proof}
The proof of Theorem \ref{thm:main} involves the following three steps.




First, we prove that $|\bp u|$ remains uniformly bounded along the flow in our scenario (Theorem \ref{thm:e''bound}). This follows from a version of the standard $\varepsilon$-regularity estimate (Proposition \ref{prop:vepsilonreg}) which is compatible with 
the split Bochner formula (Lemma \ref{l:evolution-of-e}). 
%

Next, we show (for general target manifolds) 
that a uniform $L^q$ bound on the \emph{stress-energy tensor} (\ref{stressenergy}) implies the following estimate on the (outer) energy scale at a finite-time singularity (Theorem \ref{thm:lambdaest}-Corollary \ref{cor:lambdatozero}):
\[
\lambda(t) = O(T - t)^{\frac{q}{2}} \quad (t \nearrow T).
\]
For $q > 1,$ this strengthens the standard ``type-II'' bound $\lambda(t) = o(T-t)^{\frac12}$ for (\ref{e:hm}) in dimension two. 
The improvement is crucial; indeed, it fails in Topping's counterexample (see \cite{Topping2004-mathz}, Theorem 1.14e). In the context of Theorem \ref{thm:main}, the uniform bound on $|\bp u|$ implies an $L^2$ bound on the stress-energy tensor, yielding the improved blowup rate with $q = 2$ (Corollary \ref{cor:nnhbscblowuprate}). 

Last, we use this small amount of ``spare time'' before the blowup to construct a supersolution for the angular component of $du$ in the neck region (\S \ref{ss:evolutionofangular}-\ref{ss:constructionofsuper}). Using the stress-energy bound (or the uniform bound on $|\bp u|$) once more, we obtain strong decay estimates on the full energy density 
(Theorems \ref{thm:decayest}-\ref{thm:dbaruboundeddecayest}), leading directly to our main theorems (\S \ref{sec:maintheorems}).

\subsection{Acknowledgements}
C. Song is partially supported by NSFC no. 11971400. A. Waldron is partially supported by DMS-2004661. The authors thank K. Gimre for editorial comments.

\vspace{5mm}

\section{Basic identities}\label{sec:basic}

In this section, we derive all of the identites required to obtain estimates along the flow. The formulae in \S \ref{ss:intrinsicviewpoint}-\ref{ss:bochnerformula} apply to general domains and targets, while those in \S \ref{ss:holomorphicsplitting}-\ref{ss:splitbochner} specialize to a surface domain and K{\"a}hler target. Two basic references are the reports by Eells and Lemaire \cite{EellsLemaire1995} and the textbook by Schoen and Yau \cite{SchoenYau1997}.


\subsection{Intrinsic viewpoint on harmonic map flow}\label{ss:intrinsicviewpoint}

Let $(M,g)$ and $(N,h)$ be compact Riemannian manifolds. Given a smooth map $u : M \to N,$ we write $du$ for its differential, which may be viewed as a section of
$$\E = u^* TN\otimes T^*M .$$
We shall use the notation $\LA \cdot, \cdot \RA$ for the pullback of $h$ to $u^* TN,$ combined appropriately with $g$ on tensor products with $T^*M.$
The \emph{energy density} of $u$ is given by
$$e(u) = \frac{1}{2}|du|^2 = \frac{1}{2} g^{ij}\<\p_iu, \p_ju\>,$$
and the \emph{Dirichlet functional} by
\[ E(u) = E_g(u) = \int_M e(u) \, dV_g.\]
We denote by $\nabla$ the pullback to $u^*TN$ of the Levi-Civita connection of $h,$ which we will combine with that of $g$ on tensors.

Under an infinitesimal variation
\[
\begin{split}
u & \to u + \delta u \\
g & \to g + \delta g,
\end{split}
\]
we have
\begin{equation}\label{variationofdirichlet}
\begin{split}
\de E_g(u) 
&=\frac{1}{2} \int (2g^{ij} \LA \partial_i u, \nabla_j \delta u \RA + \delta g^{ij}\<\p_iu, \p_ju\>+g^{ij}\<\p_iu, \p_ju\>\frac12g^{k \ell}\delta g_{k \ell}) dV_g \\
&= \int \left( - g^{ij} \LA \nabla_j \partial_i u, \delta u \RA + \frac12 \left( - g^{ik}g^{lj}\<\p_k u,\p_l u\>+\frac12 |d u|^2 g^{ij} \right) \delta g_{ij} \right) dV_g\\
&= - \int \left( \LA \mathcal{T}(u), \delta u \RA + \frac{1}{2} \<S,\delta g\> \right) dV_g.
\end{split}
\end{equation}
Here
\begin{equation}\label{tensionfield}
\mathcal{T}(u) = \tr_g \nabla du
\end{equation} 
is the \emph{tension field}, and
\begin{equation}\label{stressenergy}
S(u) =\< du\otimes du\> -\frac12 |d u|^2g
\end{equation}
is the \emph{stress-energy tensor}. The former is a section of $u^*TN$, and the latter of $\mathrm{Sym}^2 T^*M.$

Recall that harmonic map flow is the evolution equation
\begin{equation*}
\frac{\partial u}{\partial t} = \mathcal{T} (u)
\end{equation*}
for the map $u = u(x,t).$ 
This is the negative gradient flow of the Dirichlet functional, where the metric $g$ is fixed in time. From (\ref{variationofdirichlet}), we have the \emph{global energy identity}
\begin{equation}\label{globalenergyidentity}
E(u(t_2)) + \int_{t_1}^{t_2} \!\!\!\! \int_M |\mathcal{T} (u)|^2 \, dV_g \, dt = E(u(t_1))
\end{equation}
for a sufficiently regular solution of (\ref{e:hm}).
\begin{rmk}
In what follows, we shall always assume that $u$ solves (\ref{e:hm}) with $L^2_1$ initial data on $\LB 0, T \right),$ where $0 < T \leq \infty,$ and is smooth for $0 < t < T.$ 
Since the Struwe solution is a concatenation of finitely many such solutions, this will entail no loss of generality. 
\end{rmk}

\vspace{5mm}

\subsection{Pointwise energy identity}\label{ss:pointwiseenergyidentity} The following brief calculation yields a useful ``pointwise'' version of (\ref{globalenergyidentity}), which appears to be new. 
Using normal coordinates at a point, we take the divergence of $S:$
\begin{equation}\label{divS}
\begin{split}
\left(\div \, S\right)_j = \n_iS_{ij} & =\<\n_i\partial_iu, \partial_ju\>+\<\partial_i u, \n_i\partial_j u\>-\frac12\n_j |d u|^2 \\
& = \<\mathcal{T}(u), \partial_ju\>+\<\partial_i u, \n_j \partial_i u\>-\frac12\n_j |d u|^2\\
& = \< \mathcal{T}(u), \partial_j u \>.
\end{split}
\end{equation}
Here we have used the fact
\begin{equation}\label{partialscommute}
\nabla_i \partial_j u = \nabla_j \partial_i u,
\end{equation}
which follows from torsion-freeness of the Levi-Civita connection(s). The identity (\ref{divS}) appears in the paper of Baird and Eells \cite[(2.10)]{BairdEells1964}.

Taking another divergence, we obtain
\[ \div (\div ( S ) ) =\<\n\mathcal{T}(u), d u\>+| \mathcal{T}(u) |^2.\]
On the other hand, for a solution $u(x,t)$ of harmonic map flow, we have 
\begin{equation*}
\frac{\p e(u) }{\p t}  = \left\langle \nabla \frac{\pd u}{\pd t}, du \right\rangle = \left\langle \nabla \mathcal{T}(u), du \right\rangle.
\end{equation*}
We obtain the \emph{pointwise energy identity}:
\begin{equation}\label{pointwiseenergy}
\frac{\p e(u)}{\p t}  + |\mathcal{T}(u)|^2 = \div ( \div ( S ) ).
\end{equation}

\vspace{5mm}

\subsection{Bochner formula}\label{ss:bochnerformula} 
Let $\nabla$ (as above) be the connection induced on $\E = u^*TN\otimes T^*M $ by the pullback of the Levi-Civita connection on $N,$ coupled with that of $M.$ The connection $\nabla$ defines a (crude) Laplace operator $\Delta = - \nabla^* \nabla$ on $\E.$ We calculate:

\begin{equation}\label{prebochner}
\begin{split}
\left( \Delta du \right)_j {}^a = \nabla_i \nabla_i \p_j u^{a} & = \nabla_i \nabla_j \p_i u^a \\
& = \nabla_j \nabla_i \p_i u^{a} + \LB \nabla_i, \nabla_j \RB \p_i u^a  \\
& = \nabla_j \mathcal{T}(u)^a - {}^M \! R_{ij}{}^k{}_i \p_k u^a + (u^*{}^N \! R)_{ij}{}^a{}_{b} \p_i u^b \\
& = \nabla_j \mathcal{T}(u)^a - {}^M \! R_{ij}{}^k{}_i \p_k u^a + {}^N \! R_{c d}{}^a{}_b \p_i u^c \p_j u^d \p_i u^b \\
& = \nabla_j \mathcal{T}(u)^a + {}^M \! Ric_{jk} \p_k u^a - {}^N \! R^a {}_{bcd} \p_i u^b \p_j u^c \p_i u^d.
\end{split}
\end{equation}
Here we have used (\ref{partialscommute}) in the first line, and the Bianchi identities in the last line.

Now, given a map $u: M \times \LB 0, T \right) \to N,$ we shall write $\frac{\nabla}{\partial t}$ for the time-component of the covariant derivative induced by the pullback of the Levi-Civita connection to $M \times \LB 0, T \right).$ 
The identity
\begin{equation*}
\frac{\nabla}{\partial t} du = \nabla \frac{\p u}{\p t}
\end{equation*}
is easily checked in normal coordinates. Assuming that $u(x,t)$ is a solution of harmonic map flow, we obtain
\begin{equation}\label{evolofdu}
\frac{\nabla}{\partial t} du = \nabla \mathcal{T}(u).
\end{equation}
Then (\ref{prebochner}) becomes
\begin{equation}\label{nablauajevolution}
\left( \left( \frac{\nabla}{\partial t} - \Delta \right) du\right) {}_j {}^a = - \left( {}^M \! Ric_{ij} \right) \p_i u^a + \left( {}^N \! R^a{}_{bcd} \right) \p_i u^b \p_j u^c \p_i u^d.
\end{equation}
Taking an inner product with $du,$ and using the identity
\[
\Delta e(u) = \LA \Delta du , du \RA + |\n du|^2,
\]
we obtain
\begin{equation}
\begin{split}
& \left( \frac{\p}{\p t}-\Delta \right) e(u) = - | \n du|^2 - h_{a b} \left({}^M \! Ric^{ij} \right) \left(\p_i u^a, \p_j u^b \right) \\
& \qquad \qquad \qquad \qquad + g^{ik}g^{j\ell} \left( {}^N \! R_{a b c d} \right) \left( \p_i u^a, \p_j u^b, \p_k u^c, \p_\ell u^d \right).
\end{split}
\end{equation}
This yields the differential inequality
\begin{equation}\label{eweitz}
\left( \frac{\p}{\p t}-\Delta \right) e(u) \le -|\n d u|^2 + C_N e^2(u) + C_M e (u).
\end{equation}
Here, $C_N$ is a constant times an upper bound for the sectional curvature of $N,$ and $-C_M$ is a constant times a lower bound for the Ricci curvature of $M.$

\begin{rmk}
If the sectional curvature of $N$ is non-positive, then one may let $C_N = 0$ in (\ref{eweitz}), which gives a uniform bound on $e(u)$ by Moser's Harnack inequality. This is the key point in the proof of the Eells-Sampson Theorem \cite{EellsSampson1964}.
\end{rmk}

\vspace{5mm}

\subsection{Holomorphic splitting in the K{\"a}hler case}\label{ss:holomorphicsplitting}


We now restrict to the case that $M = \Sigma$ is an oriented 
surface 
and $N$ is a 
Hermitian manifold of complex dimension $n,$ i.e., carries an integrable almost-complex structure compatible with $h.$ 

The complexified tangent spaces decompose as
\begin{equation}\label{tcsplittings}
T \Sigma^\C = T^{1,0}\Sigma \oplus T^{0,1}\Sigma, \qquad T N^\C = T^{1,0}N \oplus T^{0,1} N.
\end{equation}
As an element of $\mathcal{E}^\C = \mathcal{E} \otimes \C,$ the differential of $u$ decomposes under the splittings (\ref{tcsplittings}) as
\begin{equation}\label{dusplitting}
du = \left( \begin{array}{cc}
\pd u & \bp u \\[2mm]
\pd \bar{u} & \overline{\pd u} \end{array} \right).
\end{equation}
Here we abuse notation slightly; 
letting $z$ be a local conformal coordinate on $\Sigma$ near a point $p,$ and $\{w^{\alpha} \}_{\alpha = 1}^n$ local holomorphic coordinates on $N$ near $u(p),$ we have
\begin{equation}
\begin{aligned}
\p u & = \frac{\partial w^\alpha}{\partial z} \, dz \otimes \frac{\pd}{\pd w^\alpha} ,  \qquad & \bp u = \frac{\pd w^\alpha}{ \pd \bar{z}}  d \bar{z} \otimes \frac{\pd}{\pd w^\alpha} \\
\p \bar{u} & = \frac{ \pd \overline{w^\alpha}}{\pd z} \, dz \otimes \frac{\pd}{\pd \bar{w}^\alpha}, \qquad & \overline{\p u} = \frac{ \pd \overline{w^\alpha } }{\pd \bar{z}} d \bar{z} \otimes \frac{\pd}{\pd \bar{w}^\alpha},
\end{aligned}
\end{equation}
where
\[
\begin{split}
dz = dx + i dy, \qquad & \frac{\p}{\p z} = \frac12 \left( \frac{\p}{\p x} - i \frac{\p}{\p y} \right) \\
d\bar{z} = dx - i dy, \qquad & \frac{\p}{\p \bar{z} } = \frac12 \left( \frac{\p}{\p x} + i \frac{\p}{\p y} \right), \quad etc.
\end{split}
\]
We shall also write
\begin{equation*}
\begin{split}
w_z & = \frac{\p w^\alpha}{\partial z} \frac{\p}{\p w^\alpha}, \quad w_{\bar{z}} = \frac{\p w^\alpha }{ \partial \bar{z} } \frac{\p}{\p w^\alpha} \\
\bar{w}_{\bar{z}} & = \overline{w_z}, \qquad \bar{w}_z = \overline{w_{\bar{z}}}.
\end{split}
\end{equation*}
Let $h(\cdot ,\cdot )$ denote the complex-linear extension of the metric to $TN^\C.$ 
Then
\begin{equation}\label{habstuff}
h\left( \frac{\p}{\p w^\alpha}, \frac{\p}{\p w^\beta} \right) = 0 = h\left( \frac{\p}{\p \bar{w}^\alpha}, \frac{\p}{\p \bar{w}^\beta} \right),
\end{equation}
and the complex $n \times n$ matrix
$$h_{\alpha \bar{\beta}} = h\left( \frac{\p}{\p w^\alpha}, \frac{\p}{\p \bar{w}^\beta} \right)$$
is Hermitian, with $h_{\alpha \bar{\beta}} = \overline{h_{\beta \bar{\alpha}}},$ and positive-definite.
Denote the Hermitian inner product corresponding to $h$ by
\begin{equation}
\LA v,w \RA = h(v, \bar{w} )
\end{equation}
for $v, w \in T N^\C,$ which agrees with the ordinary metric on the real tangent bundle $TN \subset TN^\C.$ Denote the corresponding norm by $| \cdot |.$ 

Suppose now that $N$ is K\"ahler. The K\"ahler form on $N$ can be written
\begin{equation*}
\omega_N = i h_{\alpha \bar{\beta}} dw^\alpha \wedge d\bar{w}^{\beta}.
\end{equation*}
We also extend $g$ complex-linearly, and let
\begin{equation}\label{sigma2}
\sigma^2(z) = g \left( \frac{\pd}{\pd z}, \frac{\pd}{\pd \bar{z}} \right) > 0.
\end{equation}
The metric and K\"ahler form on $\Sigma$ can be written
\begin{equation*}
g = g_\Sigma = \sigma^2 \left( dz \otimes d\bar{z} + d\bar{z} \otimes dz \right), \qquad \omega_\Sigma = i \sigma^2 dz \wedge d\bar{z}.
\end{equation*}
Let
\begin{equation}
\begin{split}
e_{\p}(u) & = |\p u|^2 = \sigma^{-2} |w_z|^2 = \sigma^{-2} h_{\alpha \bar{\beta}} w^\alpha_z \bar{w}^\beta_{\bar{z}} \\
e_{\bp}(u) & = |\bp u|^2 = \sigma^{-2} |w_{\bar{z}} |^2  = \sigma^{-2} h_{\alpha \bar{\beta}} w^\alpha_{\bar{z}} \bar{w}^\beta_z.
\end{split}
\end{equation}
Then according to (\ref{dusplitting}), the energy density decomposes as
\[e(u) =\frac12|du|^2 
= e_{\p}(u)+e_{\bp}(u).\]
On the other hand, the K{\"a}hler form on $N$ pulls back to
\[
\begin{split}
u^*\om_N & = i h \left( \left( \partial u + \bp u \right) \wedge \left( \partial \bar{u} + \overline{ \p u} \right) \right) \\
& = i \left( |w_z|^2 - |w_{\bar{z}} |^2 \right) dz \wedge d \bar{z} \\
& = (e_{\p}(u)-e_{\bp}(u)) \om_\Sigma.
\end{split}
\]
Thus
\begin{equation}\label{kappadef}
E_{\p}(u) - E_{\bp}(u) 
= \int_\Si u^*\om_N=:\ka
\end{equation}
is an invariant of the homotopy class of $u;$ the energy of $u$ satisfies
\begin{equation}\label{Esplitting}
E(u)=E_{\p}(u)+E_{\bp}(u)=2E_{\p}(u)-\ka=2E_{\bp}(u)+\ka.
\end{equation}
Hence, an (anti-)holomorphic map minimizes the Dirichlet energy within its homotopy class. It also follows from (\ref{Esplitting}) that both $E_{\p}(u(t))$ and $E_{\bp}(u(t))$ decrease in time along a solution of harmonic map flow. 

\vspace{5mm}

\subsection{Hopf differential and stress-energy tensor}\label{ss:stresshopf} We now define the \emph{Hopf differential}
$$\Phi(u) = h \! \left( du \otimes du \right) ^{2,0} \in \mathrm{Sym}^2 \Omega^{1,0}_M.$$
In local coordinates, we have
\begin{equation}\label{hopfincoordinates}
\begin{split}
\Phi(u) & = h\left( \left( \partial u + \partial \bar{u} \right) \otimes \left( \partial u + \partial \bar{u} \right) \right) \\
& = \left( h(w_z, w_z) + h(\bar{w}_z, \bar{w}_z) + h \left( w_z, \bar{w}_z  \right) + h(\bar{w}_z, w_z) \right) dz \otimes dz \\
& = 2 \LA w_z, w_{\bar{z}} \RA dz \otimes dz,
\end{split}
\end{equation}
where we have used (\ref{habstuff}).
It follows from (\ref{hopfincoordinates}) that the Hopf differential of a harmonic map is holomorphic. 
We shall not use this fact, but only the following identity which holds for general maps.

\begin{lemma}\label{lemma:stress-tensor}
For a differentiable map $u : \Si \to N,$ the stress-energy tensor is given by
\begin{equation}\label{stresshopfidentity}
 S(u)= 2 \Re \Phi (u).
\end{equation} 
In particular, the pointwise bound
\begin{equation}\label{hopfe'e''bound}
|S(u)|_g \leq 4 \sqrt{e_{\p}(u)e_{\bp}(u)}
\end{equation}
holds.
\end{lemma}
\begin{proof}
Recall the definition (\ref{stressenergy}) of $S(u).$
Since $du$ is real, and in view of (\ref{habstuff}), we have
\[\begin{aligned}
  \<du\otimes du\> = h(du \otimes du) & = h\left( \p u \otimes \o{\p u} \right) + h \left( \p u\otimes \p \bar{u} \right)
  + h\left( \o{\p u}\otimes \p u \right) + h \left( \o{\p u}\otimes \bp u \right) \\
  &\quad + h\left( \bp u \otimes \overline{\p u} \right) + h \left( \bp u \otimes \pd \bar{ u } \right)
  + h \left( \pd \bar{u} \otimes \p u \right) + h \left( \pd \bar{u} \otimes \bp u \right) \\
  &=(|w_z|^2+|w_{\bar{z}} |^2)(dz\otimes d\bar{z}+d\bar{z}\otimes dz)\\
  &\quad+2\(\<w_z, w_{\bar{z}}\>dz\otimes dz+\<w_{\bar{z}}, w_z\>d\bar{z}\otimes d\bar{z}\)\\
  &=\frac12|du|^2g+4\Re\(\<w_z, w_{\bar{z}}\>dz\otimes dz\).
\end{aligned}\]
Rearranging yields (\ref{stresshopfidentity}). Applying (\ref{hopfincoordinates}) and Cauchy-Schwarz to (\ref{stresshopfidentity}) yields (\ref{hopfe'e''bound}).
\end{proof}

\begin{rmk}
There is a strong analogy between harmonic maps from a Riemann surface to an (almost-)complex manifold and Yang-Mills connections on 4-manifolds.
Letting $j$ and $J$ be the complex structures on $\Sigma$ and $N,$ respectively, define an operator
\[
\begin{split}
& \star: \E\to \E\\
& \phi\mapsto J \circ \phi \circ j.
\end{split}
\]
Then $\star^2 = 1,$ and the differential $du$ decomposes along the $\pm$-eigenspaces of $\star$ as
\[ du= du^+ + du^- \]
where
\[
\begin{split}
du^+& = \frac12(du + J\circ du \circ j)=2 \Re \bp u, \\
du^-& = \frac12(du - J \circ du \circ j)=2 \Re \p u .
\end{split}
\]
The stress energy tensor then takes the form
\[ S=2\<du^+ \otimes du^-\>,\]
which can be compared with \cite[Remark 2.8]{Waldron2019}.

\end{rmk}

\vspace{5mm}

\subsection{Split Bochner formula}\label{ss:splitbochner} We also require a split Bochner formula 
that goes back to Schoen and Yau \cite{SchoenYau1978}, Toledo \cite{Toledo1979}, and Wood \cite{Wood1979} for harmonic maps, and has been exploited in a parabolic context by Liu and Yang \cite{LiuYang2010}.

Since $\Sigma$ and $N$ are K{\"a}hler, their complex structures 
commute with the $\C$-linear extension of $\nabla$ to $\mathcal{E}^\C.$ Hence, the operators on the LHS of (\ref{nablauajevolution}) preserve the splittings (\ref{tcsplittings}). Taking an inner product with $\bp u$ in (\ref{nablauajevolution}), we obtain
\begin{equation}\label{e''evolfirst}
\frac12\(\frac{\p}{\p t}-\Delta\)e_{\bp}(u) = - |\nabla \bp u |^2 + \LA \bp u , I \RA + \LA \bp u, II \RA,
\end{equation}
where $I$ and $II$ are the terms on the RHS of (\ref{nablauajevolution}). These may be expanded as follows. Since $Ric_\Sigma = K_\Sigma g,$ we have
\begin{equation}
\LA \bp u , I \RA = - K_\Sigma |\bp u|^2 = -K_\Sigma e_{\bp}(u).
\end{equation}
For the second term, choose a conformal coordinate $z = x + iy$ such that $\{ \frac{\p}{\p x}, \frac{\p}{\p y} \}$ form an orthonormal basis at $p$ (in particular, $\sigma^2(p) = \frac{1}{2}$). At the point $p,$ we have
\begin{equation}\label{dbaruII}
\begin{split}
\LA \bp u, II \RA & = {}^N \! R \left(\bp u \left( \frac{\p}{\p x} \right) , du \left( \frac{\p}{\p y} \right), du \left( \frac{\p}{\p x} \right) , du \left( \frac{\p}{\p y} \right) \right) \\
& \quad + {}^N \! R \left(\bp u \left( \frac{\p}{\p y} \right) , du \left( \frac{\p}{\p x} \right), du \left( \frac{\p}{\p y} \right), du \left( \frac{\p}{\p x} \right) \right).
\end{split}
\end{equation}
Since
\begin{equation*}
\frac{\p}{\p x} = \frac{\p}{\p z} + \frac{\p}{\p \bar{z}}, \qquad \frac{\p}{\p y} = i \left( \frac{\p}{\p z} - \frac{\p}{\p \bar{z}} \right)
\end{equation*}
we obtain
\begin{equation*}
\begin{split}
du\left( \frac{\p}{\p x} \right) & = \left( w_z^\alpha + w_{\bar{z}}^\alpha \right) \frac{\p}{\p w^\alpha} +  \left( \bar{w}_z^\alpha + \bar{w}_{\bar{z}}^\alpha \right) \frac{\p}{\p \bar{w}^\alpha} \\
du\left( \frac{\p}{\p y} \right) & = i \left( w_z^\alpha - w_{\bar{z}}^\alpha \right) \frac{\p}{\p w^\alpha} + i \left( \bar{w}_z^\alpha - \bar{w}_{\bar{z}}^\alpha \right) \frac{\p}{\p \bar{w}^\alpha} \\
\bp u \left( \frac{\p}{\p x} \right) & = w_{\bar{z}}^\alpha \frac{\p}{\p w^\alpha}, \qquad
\bp u \left( \frac{\p}{\p y} \right) = -i w_{\bar{z}}^\alpha \frac{\p}{\p w^\alpha}.
\end{split}
\end{equation*}
Substituting into (\ref{dbaruII}), and using the fact that ${}^N \! R$ is a section of $\Lambda^{1,1} \otimes \Lambda^{1,1},$ we obtain
\begin{equation}
\begin{split}
\LA \bp u, II \RA & = {}^N \! R_{\alpha \bar{\beta} \gamma \bar{\delta} } w^\alpha_{\bar{z} } i \left(\bar{w}^\beta_z - \bar{w}^\beta_{\bar{z}} \right) \left( w^\gamma_z + w^\gamma_{\bar{z}} \right) i \left( \bar{w}^\delta_z - \bar{w}^\delta_{\bar{z} } \right) \\
& + {}^N \! R_{\alpha \bar{\beta} \bar{\delta} \gamma} w^\alpha_{\bar{z} } i \left(\bar{w}^\beta_z - \bar{w}^\beta_{\bar{z}} \right) \left( \bar{w}^\delta_z + \bar{w}^\delta_{\bar{z}} \right) i \left( w^\gamma_z - w^\gamma_{ \bar{z} } \right) \\
& + {}^N \! R_{\alpha \bar{\beta} \gamma \bar{\delta} } (-i w^\alpha_{\bar{z} } ) \left(\bar{w}^\beta_z + \bar{w}^\beta_{\bar{z}} \right) i \left( w^\gamma_z - w^\gamma_{\bar{z}} \right) \left( \bar{w}^\delta_z + \bar{w}^\delta_{\bar{z} } \right) \\
& + {}^N \! R_{\alpha \bar{\beta} \bar{\delta} \gamma } (-i w^\alpha_{\bar{z} } ) \left(\bar{w}^\beta_z + \bar{w}^\beta_{\bar{z}} \right) i \left( \bar{w}^\delta_z - \bar{w}^\delta_{\bar{z}} \right) \left( \bar{w}^\gamma_z + \bar{w}^\gamma_{\bar{z} } \right) \\
& = 2 \, {}^N \! R_{\alpha \bar{\beta} \gamma \bar{\delta} } w^\alpha_{\bar{z} } \bar{w}^\beta_{z} \left( - \left( w^\gamma_z + w^\gamma_{\bar{z}} \right) \left( \bar{w}^\delta_z - \bar{w}^\delta_{\bar{z} } \right) + \left( w^\gamma_z - w^\gamma_{\bar{z}} \right) \left( \bar{w}^\delta_z + \bar{w}^\delta_{\bar{z} } \right) \right),
\end{split}
\end{equation}
where we have combined the first with the fourth, and the second with the third lines. This yields 
\begin{equation}
\begin{split}
\LA \bp u, II \RA & = 4 \, {}^N \! R_{\alpha \bar{\beta} \gamma \bar{\delta} } w^\alpha_{\bar{z}} \bar{w}^\beta_z \left( w^\gamma_z \bar{w}^\delta_{\bar{z}} -  w^\gamma_{\bar{z}} \bar{w}^\delta_{z} \right) \\
& = \sigma^{-4} \, {}^N \! R_{\alpha \bar{\beta} \gamma \bar{\delta} } \left( w^\alpha_{\bar{z}} \bar{w}^\beta_z w^\gamma_z \bar{w}^\delta_{\bar{z}} -  w^\alpha_{\bar{z}} \bar{w}^\beta_z w^\gamma_{\bar{z}} \bar{w}^\delta_{z} \right) =: q_1(u) + q_2(u).
\end{split}
\end{equation}
Returning to (\ref{e''evolfirst}), we obtain
\begin{equation*}\label{e''evol}
\frac12\(\frac{\p}{\p t}-\Delta\)e_{\bp}(u) = - |\nabla \bp u | - K_{\Sigma} e_{\bp}(u) + q_1(u) + q_2(u).
\end{equation*}
The condition that $h$ have nonnegative holomorphic bisectional curvature (see Goldberg and Kobayashi \cite{GoldbergKobayashi1967} or Schoen and Yau \cite{SiuYau1980}, Ch. VIII) is equivalent to the assumption that $q_1(u) \leq 0$ for any map $u.$ 
It is also clear that
\begin{equation*}
q_2(u) \leq C'_N e_{\bp}(u)^2,
\end{equation*}
where $C'_N$ is a constant times the supremum of the holomorphic sectional curvature of $N.$ 
Similar formulae hold for the holomorphic part of the energy density, $e_{\p}(u).$

We have shown:
\begin{lemma}[Liu and Yang \cite{LiuYang2010}]\label{l:evolution-of-e}
Assume that $N$ has nonnegative holomorphic bisectional curvature. For a solution $u(x,t)$ of (HM), we have
\begin{align}
  \frac12\(\frac{\p}{\p t}-\Delta\)e_{\p}(u)&\le -|\n\p u|^2-K_\Si e_{\p}(u)+C'_N e_{\p}(u)^2 \label{e:d-energy}\\
  \frac12\(\frac{\p}{\p t}-\Delta\)e_{\bp}(u)&\le -|\n\bp u|^2-K_\Si e_{\bp}(u)+C'_N e_{\bp}(u)^2. \label{e:dbar-energy}
\end{align}
\end{lemma}

\vspace{5mm}

\section{Epsilon-regularity}\label{sec:epsilonreg}

\subsection{Scalar heat inequality} 
In this subsection, we estimate a continuous 
weak solution $v(x,t)$ of the differential inequality
\begin{equation}\label{vequation}
\left( \frac{\p}{\p t} -\Delta_g \right) v \le A v^3 + B v
\end{equation}
on a ball inside the compact surface $\left( \Sigma, g \right).$\footnote{Since our results are local, the statements also apply to $\Sigma$ where $g$ has ``bounded geometry,'' i.e., uniform bounds on each derivative of the Riemann tensor and injectivity radius bounded from below.}
These estimates will lead 
to $\varepsilon$-regularity results of the form we require.

Let
$c_0 > 0$ denote a universal constant that may decrease with each subsequent appearance. The constant $R_0 > 0$ will depend on $B$ and the geometry of $\Sigma$ (and possibly on an integer $k \in \N$), and may also decrease.
We shall always assume
\begin{equation*}
0 < R < R_0,
\end{equation*}
where $R_0$ is small enough that any relevant scale-invariant Sobolev inequality on $\R^2$ also 
holds on $B_R(p)$ with a constant independent of $R$ and $p.$ We abbreviate $B_R = B_R(p),$ and for $0 < r < R,$ denote the annulus
$$U^R_r = B_R \setminus \bar{B}_r.$$
Let $\varphi \geq 0$ be a smooth cutoff function, with $\supp \varphi \subset B_R,$ $\varphi \equiv 1$ on $B_{R/2},$ and $\| \nabla \varphi \|_{L^\infty} < 4 / R;$
this may also be ensured by choosing $R_0$ sufficiently small.

\begin{lemma}[Struwe \cite{struwevariationalmethods}, Lemma III.6.7]\label{lemma:poincare} Given any $L^2_1$ function $v$ on $B_R,$ 
we have
\begin{equation}\label{poincare}
\int v^2 \varphi^2  \leq C \left( R^2 \, \int_{B_R} |\nabla v|^2 \varphi^2 + \int_{U^R_{R/2}} \!\! v^2 \, \right)
\end{equation}
\begin{equation}\label{gagliardonirenberg}
\int v^4 \varphi^2  \leq C \left( \int_{B_R} \!\! v^2 \, \right) \, \left( \int_{B_R} |\nabla v|^2 \varphi^2 + R^{-2} \int_{U^R_{R/2}} \!\! v^2 \, \right).
\end{equation}
\end{lemma}
\begin{proof} 
The estimate (\ref{poincare}) follows from the Poincar{\'e} inequality and Young's inequality:
\begin{equation*}
\begin{split}
\int v^2 \varphi^2  & \leq C R^2 \int |\nabla \left( v \varphi \right)|^2 \\
& \leq C R^2 \int \left( | \nabla v |^2 \varphi^2 + 2 \varphi v \LA \nabla v , \nabla \varphi \RA + v^2 |\nabla \varphi|^2 \right) \\
& \leq C R^2 \int \left( | \nabla v |^2 \varphi^2 + v^2 |\nabla \varphi|^2 \right) \\
&  \leq C \left( R^2 \int_{B_R} \!\! |\nabla v|^2 + \int_{U^R_{R/2}} \!\! v^2 \, \right).
\end{split}
\end{equation*}
For (\ref{gagliardonirenberg}), we refer to Lemma III.6.7 of \cite{struwevariationalmethods}. The proof there gives an inequality
\begin{equation*}
\int v^4 \varphi^2 \leq C \left( \int_{\supp \varphi} \!\! v^2 \, \right) \int \left( |\nabla v|^2 \, \varphi^2 + v^2 \, |\nabla \varphi|^2 \right).
\end{equation*}
Since $\supp \nabla \varphi \subset U^R_{R/2},$ the estimate (\ref{gagliardonirenberg}) follows.
\end{proof}

\begin{prop}[Cf. \cite{Waldron2019}, Proposition 3.4]\label{prop:spending} Let $v$ be a nonnegative continuous weak solution of (\ref{vequation})
on $B_R \times \LB 0, T \right).$ 
Suppose that
\begin{equation*}
\int_{B_R} \!\! v(0)^2 < \frac{c_0}{A}, \qquad \sup_{0 \leq t < T} \int_{U^R_{R/2}} \!\! v(t)^2 \leq \varepsilon < \frac{c_0}{A}.
\end{equation*}
Then
\begin{equation}\label{spending:est}
\int_{B_{R/2}} \!\! v(t)^2 \leq  e^{-c t/R^2} \int_{B_R} \!\! v(0)^2 + C \varepsilon \left( 1 -  e^{-c t/R^2} \right)
\end{equation}
for $0 \leq t < T.$ Here $c, c_0 > 0$ are universal constants.
\end{prop}
\begin{proof} 

Let $c_1 > 0,$ to be determined during the proof, 
and assume that
\begin{equation*}
\eps < \frac{c_1}{A}.
\end{equation*}
Define $0 < T_1 \leq T$ to
be the maximal time such that
\begin{equation}\label{spending:assn}
\int_{B_R} v(t)^2 < \frac{c_1}{A} \quad \quad \left( \forall \,\,\,\, 0 \leq t < T_1 \right).
\end{equation}
Let $\varphi$ be as above. 
Multiplying (\ref{vequation}) by $\varphi^2 v$ and integrating by parts, we obtain
\begin{equation*}
\frac{1}{2} \frac{d}{dt} \left(\int \varphi^2 v^2\right) + \int \nabla(\varphi^2 v) \cdot \nabla v \leq A\int \varphi^2 v^{4} + B \int \varphi^2 v^2.
\end{equation*}
Applying Young's inequality, we obtain
\begin{equation*}
\begin{split}
\frac{1}{2} \frac{d}{dt} \left( \int \varphi^2 v^2 \right) + \frac{1}{2}\int \varphi^2 |\nabla v|^2  \leq 2 \int |\nabla \varphi |^2 v^2 + A \int \varphi^2 v^4 + B \int \varphi^2 v^2.
\end{split}
\end{equation*} 
For $0 \leq t < T_1,$ we apply (\ref{gagliardonirenberg}) on the RHS and rearrange, to obtain
\begin{equation*}
\frac{1}{2} \frac{d}{dt} \int \left(\varphi v\right)^2 + \left( \frac{1}{2} - C A \frac{c_1}{A} \right) \int \varphi^2 |\nabla v|^2 \leq C \left( 1 + A \frac{c_1}{A} \right) R^{-2} \varepsilon + B \int \left( \varphi v \right)^2.
\end{equation*}
Applying (\ref{poincare}) on the left-hand side, and rearranging again, we obtain
\begin{equation}\label{spending:diffineq}
\frac{d}{dt} \int \left(\varphi v\right)^2 + \frac{c }{R^2} \int (\varphi v)^2  \leq C \left(1 + c_1 \right) R^{-2} \varepsilon,
\end{equation}
where
$$c = \frac{1}{C} \left( \frac{1}{2} - C c_1 \right) - C B R^2.$$
We may assume
\begin{equation*}
c_1 < \frac{1}{4C}, \quad R_0 < \frac{1}{4CB},
\end{equation*}
so that $c > (8C)^{-1}.$

Rewrite (\ref{spending:diffineq}) as
$$\frac{d}{dt} \left( e^{c t / R^2} \int \left(\varphi v\right)^2\right) \leq C R^{-2} \varepsilon e^{c t / R^2}$$
and integrate in time, to obtain
\begin{equation}\label{spending:preest}
\begin{split}
\int_{B_{R/2}} v(t)^2 & \leq e^{- c t / R^2} \int_{B_R} v(0)^2 + C \varepsilon \left(1 - e^{- c t / R^2} \right),
\end{split}
\end{equation}
which holds 
for $0 \leq t < T_1.$

Assume, for the sake of contradiction, that $T_1 < T.$ For $t < T_1,$ (\ref{spending:preest}) implies
\begin{equation*}
\begin{split}
\int_{B_{R/2}} v^2 \leq \frac{c_0}{A} + C \varepsilon.
\end{split}
\end{equation*}
Provided that
$$2\left(2 + C\right) c_0 < c_1,$$
we have
\begin{equation}\label{spending:lastest}
\int_{B_R} v^2 = \int_{B_{R/2}} v^2 + \int_{U_{R/2}^{R}} v^2 \leq \frac{c_0}{A} + \left(1 + C\right) \varepsilon \leq \frac{(2 + C ) c_0}{A} \leq \frac{c_1}{2A}.
\end{equation}
Since $v$ is continuous for $t < T,$ the bound (\ref{spending:lastest}) persists at $t = T_1.$ This contradicts the maximality of $T_1$ subject to the open condition (\ref{spending:assn}); hence $T_1 = T,$ and (\ref{spending:est}) follows from (\ref{spending:preest}).
\end{proof}

\begin{prop}\label{prop:vepsilonreg}
Let $0< R < \min \LB R_0, \sqrt{T} \RB.$ 
Suppose that $v$ is a nonnegative continuous weak solution of (\ref{vequation}) on $B_R \times \LB 0, T \right),$ with
\begin{equation}\label{vepsilonreg:assn}
 \int_{B_R}\!\!  v(0)^2 + \sup_{0 \leq t < T} \int_{U^R_{R/2}} \!\! v(t)^2 \leq \varepsilon < \frac{c_0}{A}.
\end{equation}
Then
\begin{equation}\label{vepsilonreg:est}
\sup_{B_{R/2} \times \LB R^2 / 2 , T \right)} v(x,t) \le \frac{C \sqrt{\varepsilon}}{R}.
\end{equation}
\end{prop}
\begin{proof}  
By Proposition \ref{prop:spending}, the assumption (\ref{vepsilonreg:assn}) implies
\begin{equation}\label{vepsilonreg:rescassn}
 \int_{B_R} \!\! v(t)^2 \leq C \varepsilon
\end{equation}
for all $0 \leq t < T.$ 
This allows us to  repeat the argument of \cite[Proposition 2.3]{Waldron2016}, originally due to Schoen and Uhlenbeck \cite{SchoenUhlenbeck1982}.

Let $\tau \in \LB 0, T - R^2 \right)$ be arbitrary. Note that to establish (\ref{vepsilonreg:est}), it suffices to prove the estimate
\begin{equation}\label{vepsilonreg:reallydumbest}
\sup_{B_{R/2} \times \LB \tau + R^2 / 2 , \tau + R^2 \RB} v(x,t) \le \frac{C \sqrt{\varepsilon}}{R}.
\end{equation}
After rescaling
\begin{equation}\label{rescaling}
v(x,t) \to R \, v(Rx, R^2 t + \tau)
\end{equation}
in geodesic coordinates on $B_R,$ we may assume that $v$ is defined and continuous on $B_1 \times \LB 0 , 1 \RB.$ Then (\ref{vepsilonreg:rescassn}) is preserved, and (\ref{vequation}) becomes
\begin{equation}\label{vrescevol}
\begin{split}
\left( \frac{\partial}{ \partial t} - \Delta\right) v & \leq A v^{3} + R^2 B v \\
& \leq (A v^2 + 1) \cdot v,
\end{split}
\end{equation}
provided that $R_0 \leq 1/ \sqrt{B}.$
For $R < R_0,$ the rescaled metric (and any further rescaling) is close enough to Euclidean that we have uniform volume bounds and a uniform Sobolev constant, allowing us to apply Moser iteration with a uniform constant. 

Let $P_r(x, t) = B_r(x) \times \LB t- r^2, t \RB,$ and write $P_r = P_r(0,1).$ Define
\begin{equation}\label{rescenergy}
e(r) = \left(1 - r\right) \sup_{P_r} v,
\end{equation}
and let $e_0, r_0$ be such that
\begin{equation}\label{e0}
e_0 = e(r_0) = \sup_{0 \leq r \leq 1} e(r).
\end{equation}
Note that if $v \not \equiv 0$ then $r_0 < 1$ by definition. 
We may choose $(x_0, t_0) \in P_{r_0}$ such that
$$v(x_0, t_0) = \sup_{P_{r_0}} v = \frac{e_0}{1 - r_0}.$$ 
Now assume, for the sake of contradiction, that 
\begin{equation*}
e_0 > \frac{1}{\sqrt{A} }.
\end{equation*}
Let
$$\rho_1 := \frac{ 1 - r_0 }{2 e_0 \sqrt{ A }} < \frac{1 - r_0}{2}.$$
By definition, we have
\begin{equation}\label{supboundtorescale}
\rho_1 v(x_0, t_0) = \frac{1}{2 \sqrt{A}}. 
\end{equation}
Since $P_{\rho_1}(x_0, t_0) \subset P_{r_0 + \rho_1},$ we also have
\begin{equation*}
(1 - r_0 - \rho_1) \sup_{P_{\rho_1}(x_0, t_0)} v \leq e(r_0 + \rho_1) \leq e_0.
\end{equation*}
But notice that
$$e_0 \sqrt{A} \cdot \rho_1 = \frac{1-r_0}{2} \leq 1 - r_0 - \rho_1.$$
We therefore have
\begin{equation}\label{boundtorescale}
\rho_1 \sup_{P_{\rho_1}(x_0, t_0)} v \leq \frac{1}{\sqrt{A} }.
\end{equation}
Now define a function on $P_1$ by
$$v_1(x, t) =  \rho_1 v \left( \rho_1 x + x_0 , \rho_1^2 \left( t - 1 \right) + t_0 \right).$$
Then 
(\ref{supboundtorescale}-\ref{boundtorescale}) become 
$$v_1(0,0) = \frac{1}{2 \sqrt{A}}, \quad \sup_{P_1} v_1(x, t) \leq \frac{1}{\sqrt{A}},$$
and (\ref{vrescevol}) gives
\begin{equation}\label{vrescvolnew}
\begin{split}
\left( \frac{\partial}{ \partial t}  - \Delta \right) v_1 \leq 2 v_1.
\end{split}
\end{equation}
But then Moser's Harnack inequality (see e.g \cite[Lemma 2.2]{Waldron2016}), applied to (\ref{vrescvolnew}), gives
\begin{equation}
\begin{split}
\frac{1}{4A} = v_1(0,1)^2 & \leq C \int_0^1 \!\!\!\! \int_{B_1} v^2_1(x, t) \, dV dt \\
& \leq C \varepsilon \\
& \leq \frac{Cc_0}{A}.
\end{split}
\end{equation}
Provided that $Cc_0 < 1/4,$ this is a contradiction; therefore
$$e_0 \leq \frac{1}{\sqrt{A} }.$$
Now, from (\ref{rescenergy}-\ref{e0}), we have for any $0 < r < 1$
$$\sup_{P_r}v = \frac{e(r)}{1-r} \leq \frac{e_0}{1-r} \leq  \frac{1}{ \sqrt{A} ( 1-r ) }.$$
In particular, we have
$$\sup_{P_{3/4}} v \leq \frac{4}{\sqrt{A} }.$$
Then (\ref{vrescevol}) becomes
\begin{equation*}
\begin{split}
\left( \frac{\partial}{ \partial t}  - \Delta\right) v \leq 17 v
\end{split}
\end{equation*}
on $P_{3/4}.$ We may again apply Moser's Harnack inequality, to find
\begin{equation}\label{moserresult}
\begin{split}
\sup_{P_{1/2}} v & \leq C \left( \int_0^1 \!\!\!\! \int_{B_{3/4}} \! v^2 \, d V dt \right)^{1/2} \\
& \leq C \sqrt{\varepsilon}.
\end{split}
\end{equation}
The estimate (\ref{vepsilonreg:reallydumbest}) now follows from (\ref{moserresult}) by undoing the rescaling (\ref{rescaling}). Since $\tau$ was arbitrary, the desired estimate (\ref{vepsilonreg:est}) also follows.
\end{proof}

\vspace{5mm}

\subsection{Epsilon-regularity theorems}

The above estimates 
yield the following version of the standard $\varepsilon$-regularity theorem, as well as a split version in the K{\"a}hler case. 


\begin{thm}\label{thm:epsilonreg} Let $k \in \N.$ Given a compact Riemannian surface $\Sigma$ and a compact Riemannian manifold $N,$ there exist $\varepsilon_0 > 0,$ depending on the geometry of $N,$ and $R_0 > 0,$ depending on $k$ and the geometry of $\Sigma,$\footnote{Alternatively, one may take $R_0 = \inj_p(\Sigma) $ and let the constants in (\ref{epsilonreg:0thest}-\ref{epsilonreg:tensionfield}) depend on the geometry of $\Sigma.$}
 as follows.

Let $0 < R < \min \LB R_0, \sqrt{T} \RB$ and $0 \leq \tau \leq T - R^2.$
Suppose that $u : B_R \times \LB 0, T \right) \to N$ is a solution of (\ref{e:hm}) such that
\begin{equation}\label{epsilonreg:assn}
E(u(\tau), B_R) + \sup_{\tau \leq t < T} E \left( u(t), U^R_{R/2} \right) \leq \varepsilon < \varepsilon_0.
\end{equation}
Then
\begin{equation}\label{epsilonreg:0thest}
\sup_{B_{R/2} \times \LB \tau + R^2 / 2 , T \right)} |du| \le \frac{C  \sqrt{\varepsilon}}{R}.
\end{equation}
We further have
\begin{equation}\label{epsilonreg:kthest}
\sup_{B_{R/2} \times \LB \tau + R^2 / 2 , T \right)} | \nabla^{(k)} du | \le \frac{C_{N, k} \sqrt{\varepsilon}}{ R^{1 + k} }
\end{equation}
and
\begin{equation}\label{epsilonreg:tensionfield}
\sup_{B_{R/2} \times \LB \tau + R^2 / 2 , T \right)} | \nabla^{(k)} \mathcal{T}(u) | \le \frac{C_{N, k} }{R^{2 + k}} \sqrt{\int_\tau^T \!\!\!\!\! \int_{B_R} |\mathcal{T}(u)|^2 \, dV dt}.
\end{equation}
Here, we may take $\varepsilon_0 = \frac{c_0}{\sup_N |K_N|},$ where $K_N$ is the sectional curvature of $N$ and $c_0$ is the universal constant of Proposition \ref{prop:vepsilonreg}.
\end{thm}
\begin{proof} By applying Kato's inequality $|\nabla du| \leq |d |du||$ to (\ref{eweitz}), we obtain a differential inequality of the form (\ref{vequation}), with $v = |du|$ and $A = C_N.$ The estimate (\ref{epsilonreg:0thest}) follows directly from Proposition \ref{prop:vepsilonreg}. The derivative estimates (\ref{epsilonreg:kthest}) then follow by a standard bootstrapping argument applied to (\ref{nablauajevolution}).

The estimate (\ref{epsilonreg:tensionfield}) on the tension field can be obtained from the evolution equation 
\begin{equation*}
\begin{split}
\frac{\nabla}{\partial t}\mathcal{T}(u)&= tr_g \frac{\nabla}{\partial t} \n d u \\
& = tr_g \n \frac{\nabla }{\partial t} du + tr_g {}^N \! R \left( \frac{\partial u}{\partial t}, d u \right) du \\
& = \Delta \mathcal{T}(u) + tr_g {}^N \! R(\mathcal{T}(u), d u) d u,
\end{split}
\end{equation*}
by (\ref{evolofdu}).
Taking an inner product with $\mathcal{T}(u),$ we obtain
\begin{equation}\label{bochnerforT}
\begin{split}
\frac12 \left( \frac{\pd}{\pd t} - \Delta \right) |\mathcal{T}(u)|^2 & = -|\nabla \mathcal{T}(u)|^2 + {}^N \! R \# \mathcal{T}(u) \# \mathcal{T}(u) \# du \# du.
\end{split}
\end{equation}
By rescaling, we may assume without loss that $R = 1.$ Then we have $|du| \leq C \sqrt{\varepsilon} $ by (\ref{epsilonreg:0thest}), and (\ref{bochnerforT}) gives an inequality
\begin{equation*}
\begin{split}
\frac12 \left( \frac{\pd}{\pd t} - \Delta \right) |\mathcal{T}(u)|^2 & \leq C |\mathcal{T}(u)|^2.
\end{split}
\end{equation*}
The estimate (\ref{epsilonreg:tensionfield}) for $k = 0$ now also follows from Moser's Harnack inequality, and the higher derivative estimates are obtained by bootstrapping.
\end{proof}

\begin{thm}\label{thm:e''bound} 
Supposing that $N$ is compact K{\"a}hler with nonnegative holomorphic bisectional curvature, there exists $\delta_0 > 0$ as follows.
 
Let $0 < R < \min \LB R_0, \sqrt{T} \RB$ and $0 \leq \tau \leq T - R^2.$
Suppose that $u : B_R \times \LB 0, T \right) \to N$ is a solution of (\ref{e:hm}) such that
\begin{equation*}
E_{\bp} (u(\tau), B_R) + \sup_{\tau \leq t < T} E_{\bp } \left( u(t), U^R_{R/2} \right) \leq \delta < \delta_0.
\end{equation*}
Then
\begin{equation*}
\sup_{B_{R/2} \times \LB \tau + R^2 / 2 , T \right)} |\bp u| \le \frac{C \sqrt{\delta}}{R}.
\end{equation*}
A similar result holds for $\partial u,$ with $E_{\p}$ in place of $E_{\bp}.$

Here, we may take $\delta_0 = \frac{c_0}{\sup_N |H_N|},$ where $H_N$ is the holomorphic sectional curvature of $N$ and $c_0$ is the universal constant of Proposition \ref{prop:vepsilonreg}.
\end{thm}
\begin{proof}
By applying Kato's inequality to (\ref{e:dbar-energy}), we obtain an inequality of the form (\ref{vequation}), with $v = |\bp u|$ and $A = C_N'.$ The result then follows from Proposition \ref{prop:vepsilonreg}. 
\end{proof}

\vspace{5mm}

\section{Estimate on the energy scale}\label{sec:curvaturescale}

Let $q \geq 1.$ We now show that a uniform $L^q$ bound on the stress-energy tensor
\begin{equation}\label{stressbound}
\sup_{\tau \leq t < T} \| S(u(t)) \|_{L^q} \leq \sigma 
\end{equation}
yields control over the energy scale along the flow (see Definition \ref{defn:energyscale} below), 
based on the following lemma. Here, as in the previous section, $R_0 > 0$ will be a constant depending on the geometry of $\Sigma.$ 


\begin{lemma}\label{lemma:mainenergyest} Given $p\in \Sigma$ and $0 < R < R_0,$ write $B_R = B_R(p).$ 
Assuming the stress-energy bound (\ref{stressbound}),
for $\tau \leq t_1 \leq t_2 < T,$ we have
\begin{equation*}
E\left( u(t_2), B_{R/2} \right) \leq E\left( u(t_1), B_R\right) + C \sigma \frac{t_2 - t_1 }{R^{\frac{2}{q}}}. 
\end{equation*}
\end{lemma}
\begin{proof} Choose a cutoff $\varphi$ for $B_{R/2} \subset B_R.$ Integrating by parts against the pointwise energy identity (\ref{pointwiseenergy}), we have
\begin{equation}\label{mainenergyest:1}
\begin{split}
\int e(t_2) \varphi \, dV - \int e(t_1) \varphi \, dV + \int_{t_1}^{t_2} \!\!\!\! \int |\T|^2 \varphi & \, dV dt \\
& = \int_{t_1}^{t_2} \!\!\!\! \int \nabla^i \nabla^j \varphi S_{ij} \, dV dt \\
& \leq C (t_2 - t_1 ) \| \nabla^{(2)} \varphi \|_{L^{\frac{q}{q-1} }} \sup_{t_1 \leq t \leq t_2} \|S(t)\|_{L^q},
\end{split}
\end{equation}
where we have applied H\"older's inequality. We have
\begin{equation*}
\| \nabla^{(2)} \varphi \|_{L^{\frac{q}{q-1} }} \leq \LB \int_{B_R} \left( \frac{C}{R^2} \right)^{\frac{q}{q-1}} \, dV \RB^{\frac{q-1}{q}} \leq C R^{ \left( 2 - \frac{2q}{q-1} \right) \frac{q-1}{q} } = C R^{\frac{-2}{q}}.
\end{equation*}
Then (\ref{mainenergyest:1}) yields
\begin{equation}
E\left( u(t_2), B_{R/2} \right) - E\left( u(t_1), B_R\right) \leq C (t_2 - t_1 ) R^{\frac{-2}{q}} \sigma
\end{equation}
as desired.
\end{proof}

\begin{defn}\label{defn:energyscale}
Let $p\in \Sigma,$ $0 < \rho < R_0,$ and $\varepsilon > 0.$ Given an $L^2_1$ map $u: B_{\rho}(p) \to N,$ define the \emph{(outer) energy scale} $\ld_{\varepsilon, \rho, p}(u)$ to be the minimal number $\ld \in [0,\rho]$ such that
\begin{equation}\label{lambdadef}
\sup_{\ld<r<\rho} E\(u,U_{r/2}^{r}(p)\) < \varepsilon.
\end{equation}
Notice that if no $\lambda < \rho$ satisfies (\ref{lambdadef}), then $\lambda = \rho$ does so vacuously.
\end{defn}

\begin{lemma}\label{lemma:lambdazero} 
Suppose that $0 < \eps < \eps_0.$ 
For a solution $u: \Sigma \times \LB 0, T \right) \to N$ of (\ref{e:hm}), 
$du$ is bounded near $(p,T)$ if and only if there exists $\rho > 0$ such that $\lambda_{\varepsilon, \rho, p}(u(t)) = 0$ for all $T - \rho^2 \leq t < T.$  
\end{lemma}
\begin{proof} 
If $du$ is locally bounded, then it is clear from the definition that $\lambda_{\varepsilon, \rho, p}(u(t)) = 0$ for $\rho$ sufficienctly small. Conversely, suppose that $\ld_{\varepsilon, \rho, p}(u(t)) = 0$ for some $\rho > 0$ and all $T - \rho^2 \leq t < T.$ Then, since $du(T - \rho^2)$ is square-integrable, we may choose $R > 0$ such that
$$E(u(T - \rho^2), B_R) \leq \varepsilon.$$
But, by assumption, we also have
$$\sup_{T - \rho^2 \leq t < T} E \left( u(t), U^R_{R/2} \right) < \varepsilon.$$
Hence (\ref{epsilonreg:assn}) is satisfied, and Theorem \ref{thm:epsilonreg} implies that the energy density is bounded near $p$ as $t \nearrow T.$
\end{proof}

\begin{thm}\label{thm:lambdaest} Let $u$ be as above, and assume the stress-energy bound (\ref{stressbound}). 
Suppose that $0 < R \leq \rho < R_0$ and $\tau \leq t_1 \leq  t_2 < T$
are such that
\begin{equation}\label{lambdaest:mainassn}
\sup_{t_1 \leq t \leq t_2} E \left( u(t), B_{\rho}(p) \right) \leq E \left( u(t_2), B_{R}(p) \right) + \frac{\varepsilon}{2}.
\end{equation}
Then for $t_1 \leq t \leq t_2,$ we have
\begin{equation}\label{lambdaest:mainest}
\lambda_{\varepsilon, \rho, p }(u(t)) \leq \max \LB 4 R, \left( \frac{ C\sigma (t_2 - t) }{ \varepsilon } \right)^{\frac{q}{2}} \RB.
\end{equation}
\end{thm}

\begin{center}
\includegraphics[scale=.4]{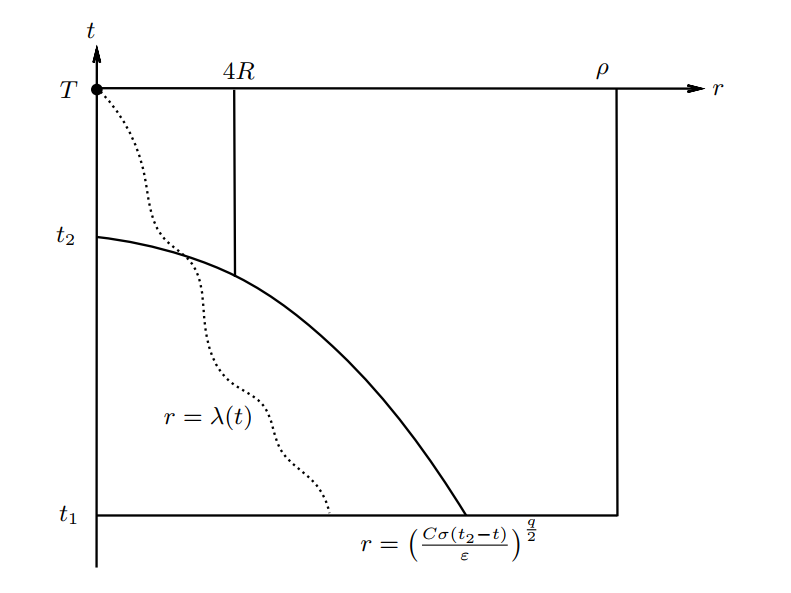}
\end{center}
 
\begin{proof} Suppose that (\ref{lambdaest:mainassn}) holds but (\ref{lambdaest:mainest}) fails, i.e., there exists $t \in \LB t_1 , t_2 \RB$ such that $\lambda = \lambda_{\varepsilon, \rho, p } \left( u(t) \right) \geq 4R$ satisfies
\begin{equation}\label{lambdaest:1}
\lambda^{\frac{2}{q}} > \frac{ 2 C(t_2 - t) \sigma }{\varepsilon}.
\end{equation}
Then, by choice of $\lambda,$ we must have
\begin{equation}\label{etulambdaleps}
E\(u(t),U_{\lambda/2}^{\lambda}\) \geq \varepsilon.
\end{equation}
Now (\ref{lambdaest:mainassn}) and (\ref{etulambdaleps}) imply
\begin{equation*}
\begin{split}
E(u(t), B_{\lambda/2}) & \leq E \left( u(t), B_{\rho} \right) - E \left(u(t), U^{\lambda}_{\lambda/2} \right) \\
& \leq E \left( u(t_2), B_R \right) + \frac{\varepsilon}{2} - \varepsilon \\
& \leq E \left( u(t_2), B_R \right) - \frac{\varepsilon}{2}.
\end{split}
\end{equation*}
But then Lemma \ref{lemma:mainenergyest} and (\ref{lambdaest:1}) give
\begin{equation*}
\begin{split}
E( u(t_2), B_R) \leq E \left( u(t_2), B_{\lambda / 4} \right) & \leq E(u(t), B_{\lambda/2} ) +C \sigma \frac{t_2 - t }{ \lambda^{\frac{2}{q}} } \\
& < E \left( u(t_2), B_R \right) - \frac{\varepsilon}{2} + \frac{\varepsilon}{2} = E\left( u(t_2), B_R \right),
\end{split}
\end{equation*}
which is a contradiction. This establishes the desired estimate (\ref{lambdaest:mainest}).
\end{proof}

\begin{cor}\label{cor:lambdatozero} Suppose that the stress-energy bound (\ref{stressbound}) holds for some $q \geq 1,$ and assume $T < \infty.$ Given any $p \in \Sigma$ and $\eps > 0,$ for $\rho > 0$ sufficiently small, we have
\begin{equation}\label{lambdaest:est}
\lambda_{\varepsilon, \rho, p}(u(t)) = O(T - t)^{\frac{q}{2}} \quad (t \nearrow T).
\end{equation}
\end{cor}
\begin{proof} 
Define
\begin{equation*}
E_0 = \lim_{r \searrow 0} \limsup_{t \nearrow T} E(u(t), B_r) \leq E(u(0)) < \infty. 
\end{equation*}
Here, the outer limit must exist because the inner $\limsup$ is nonnegative and decreasing with respect to $r.$ 
In particular, for $0 < \rho < R_0$ sufficiently small, 
there exists $t_1 < T$ such that
\begin{equation}\label{eoepso41}
\sup_{t_1 \leq t < T} E(u(t), B_\rho ) \leq E_0 + \frac{\varepsilon}{4}.
\end{equation}
On the other hand, given 
any $R > 0,$ we may also choose $t_2 \in \LB t_1, T \right)$ such that
\begin{equation}\label{eoepso42}
E(u(t_2), B_R) \geq E_0 - \frac{\varepsilon}{4}.
\end{equation}
Combining (\ref{eoepso41}) and (\ref{eoepso42}), we obtain
\begin{equation*}
\sup_{t_1 \leq t < T} E(u(t), B_\rho ) \leq E(u(t_2), B_R) + \frac{\varepsilon}{2},
\end{equation*}
which implies (\ref{lambdaest:mainassn}). Theorem \ref{thm:lambdaest} now yields a bound of the form (\ref{lambdaest:mainest}); since $R > 0$ was arbitrary, this implies (\ref{lambdaest:est}).
\end{proof}

\begin{cor}\label{cor:nnhbscblowuprate} Suppose that $N$ is compact K\"ahler with nonnegative holomorphic bisectional curvature. Assume that $p \in \Sigma,$ $0 \leq \tau < T,$ and $0 < \rho_1 < \min \LB R_0, \sqrt{T - \tau} \RB$ are such that 
\begin{equation}\label{nnhbscblowuprate:assn}
\sup_{\tau \leq t < T} E_{\bp}(u(t), B_{\rho_1}(p) ) \leq \delta < \delta_0,
\end{equation}
where $\delta_0$ is the constant of Theorem \ref{thm:e''bound}.
Then, for each $\tau + \frac{\rho_1^2}{2} \leq t < T,$ we have
\begin{equation}\label{nnhbscblowuprate:Sest}
\begin{split}
\norm{S(t)}_{L^2 ( B_{ \frac{\rho_1}{2} }(p) ) } & \leq \frac{C \sqrt{\delta E_\p(u(\tau))} }{\rho_1}.
\end{split}
\end{equation}
In particular, for each $\varepsilon > 0,$ 
there exists $\rho > 0$ such that
\begin{equation}\label{nnhbscblowuprate:Oest}
\lambda_{\varepsilon, \rho, p}(u(t)) = O(T - t) \quad (t \nearrow T).
\end{equation}
The same result holds after reversing the roles of $E_\dbar$ by $E_\p.$
\end{cor}
\begin{proof} From Lemma \ref{lemma:stress-tensor} and Theorem \ref{thm:e''bound}, we have the pointwise bound
$$|S(u(t))|^2 \leq C e_{\bp}(t) e_{\p}(t)  \leq \frac{ C\delta e_{\p}(t)  }{\rho^2_1} $$
on $B_{ \frac{\rho_1}{2} } \times \LB \tau + \frac{\rho_1^2}{2} , T \right).$ Integrating, we have
\begin{equation*}
\begin{split}
\norm{S(u(t))}_{L^2\left( B_{ \frac{\rho_1}{2} } \right)}^2 & \leq \frac{ C \delta E_{\p}(u(t)) }{\rho^2_1} \leq \frac{ C \delta E_{\p}(u(0)) }{\rho^2_1},
\end{split}
\end{equation*}
since $E_{\p}(u(t))$ is decreasing, which is (\ref{nnhbscblowuprate:Sest}). 

The estimate (\ref{nnhbscblowuprate:Oest}) now follows from the previous corollary, with $q = 2.$
\end{proof}

\begin{rmk}\label{rmk:rotsymcase}
Angenent, Hulshof, and Matano \cite{AHM2009} first established the bound (\ref{nnhbscblowuprate:Oest}) in the degree 1 rotationally symmetric case (with $N = S^2$) originally studied by Chang, Ding, and Ye \cite{CDY1992}. Raphael and Schweyer \cite{RaphaelSchweyer2013, RaphaelSchweyer2014} have shown that blowup occurs in that setting with the precise rate
\begin{equation}\label{CDYblowuprate}
\lambda(t) \sim \kappa \frac{(T - t)^\ell}{\left| \log (T - t) \right|^{2\ell/(2\ell-1)}}
\end{equation}
for each positive integer $\ell,$
as predicted by Van den Berg, Hulshof, and King \cite{BHK2003}. D\'avila, del Pino, and Wei~\cite{DPW2020} have recently constructed more extensive examples with the $\ell = 1$ rate. 
\end{rmk}


\vspace{5mm}

\section{Decay estimate in annular regions}\label{sec:decay}

In this section, we obtain spatial decay estimates for the energy density of a solution of (\ref{e:hm})
with small energy on an annulus. 
The results are less precise than those of \cite[\S 5]{Waldron2019}, but the proof is much simpler.

\subsection{Evolution of the angular energy}\label{ss:angularenergy}\label{ss:evolutionofangular} Given $0 < R < 1,$ let
\begin{equation}
U = U_R^1 = \{ (r, \theta) \mid R \leq r \leq 1 \} \subset \R^2
\end{equation}
denote an annulus in the plane. Let $g$ be a metric on $U$ of the form
\begin{equation}\label{e:4-metric}
g=\zeta^2 (dr^2+r^2d\th^2),
\end{equation}
where $\zeta = \zeta(r, \theta)$ is a smooth function on $U$ with 
\begin{equation}\label{e:5-metric}
r|d \zeta| + | \zeta^2 - 1 | \leq \alpha r^2
\end{equation} 
for a constant $ 0 \leq \alpha \leq 1/2.$

For convenience, we shall work below in cylindrical coordinates $(s=\ln r,\th).$ Letting
$$g_0=ds^2+d\th^2$$
denote the flat cylindrical metric, we have $g=\zeta^{2} e^{2s} g_0.$
The differential of $u$ is given by
\begin{equation}
du = u_s ds + u_\theta d\theta.
\end{equation}
The tension field with respect to $g_0$ is given by
\[ \mathcal{T}_0(u)=\n_s u_s+\n_\th u_\th,\]
where $\n$ denotes the pullback connection on $u^*TN,$ as above.
The flow equation (\ref{e:hm}) with respect to the metric $g$ becomes
\begin{equation}\label{e:hm-cylin}
\p_t u = \mathcal{T}(u) = \zeta^{-2}e^{-2s}\mathcal{T}_0(u).
\end{equation}

\begin{lemma}\label{l:equ-angular} There exists $\varepsilon_0 > 0,$ depending on the geometry of $N,$ as follows.

Let $u$ be a solution of (\ref{e:hm}) on $U \times [0, T)$ with respect to a metric $g$ given by (\ref{e:4-metric}). Suppose that for some $0 < \eta^2 < \ep_0,$ 
 we have
\begin{equation}\label{e:small-assumption} 
r|du|_g+ r^2|\n du|_g + r^3|\n^2 du|_g \le \eta. 
\end{equation}
Then the angular energy
\begin{equation}\label{angularenergy}
f = f(u;r, t) := \sqrt{ \int_{ \{r\} \times S^1 } |u_\th (r,\theta,t)|^2 d\th }
\end{equation}
satisfies a differential equality
\begin{equation}\label{e:angular-energy}
\p_tf -\(\p_r^2+\frac{1}{r}\p_r - \frac{1-C \eta} {r^2}\) f \le C \alpha \eta.
\end{equation}
\end{lemma}
\begin{proof}

We start from the identity
\[ \frac12 \p_s^2f^2= \int_{S^1}\( |\n_s u_{\th}|^2+\< \n_s^2u_{\th}, u_\th\>\).\]
By interchanging derivatives and integrating by parts, we compute
\[\begin{aligned}
\frac12 \p_s^2f^2
&=\int_{S^1}\(|\n_s u_{\th}|^2 +\<\n_\th \n_s u_{s}+R(u_s, u_\th)u_s,u_{\th}\>\)\\
&=\int_{S^1}\(|\n_s u_{\th}|^2 - \<\n_su_s, \n_\th u_\th\>+\<R(u_s, u_\th)u_s,u_{\th}\>\).
\end{aligned}\]
Using the flow equation (\ref{e:hm-cylin}), we get
\begin{equation*}
\frac12 \p_s^2f^2
=\int_{S^1}\(|\n_s u_{\th}|^2 + |\n_\th u_\th|^2 - \zeta^2 e^{2s}\<\p_t u, \n_\th u_{\th}\>+\<R(u_s, u_\th)u_s,u_{\th}\>\)
\end{equation*}
On the other hand, 
we have
\begin{equation*}
\frac12 \p_tf^2=  \int_{S^1}\<\n_tu_\th, u_\th\>d\th
= -\int_{S^1}\<\p_t u, \n_\th u_{\th}\>.
\end{equation*}
Combining the above two identities, we get
\begin{equation}\label{e:4-1}
  \frac12(e^{2s}\p_t-\p_s^2)f^2=- \( \int_{S^1}|\n_s u_{\th}|^2+|\n_{\th} u_{\th}|^2 +\<R(u_s, u_\th)u_s,u_{\th}\>\) + Q(u),
\end{equation}
where
\[ Q(u)=\int_{S^1}(\zeta^2-1)e^{2s}\<\p_t u, \n_\th u_{\th}\>.\]
Now, using (\ref{e:4-metric}) and plugging in the equation (\ref{e:hm-cylin}) again, we have
\begin{equation*}
 |Q(u)|\le \sup_{S^1}\left|\zeta^{-2}(\zeta^2-1)\right|\cdot \left|\int_{S^1}\<\zeta^{2}e^{2s}\p_tu, \n_\th u_\th\>\right|
 \le C \alpha e^{2s}\left|\int_{S^1}\<\mathcal{T}_0(u), \n_\th u_\th\>\right|.
\end{equation*}
Integrating by parts and using H\"older's inequality, we get 
\begin{equation}\label{e:4-2}
  |Q(u)|\le 2 \alpha e^{2s} \left| \int_{S^1} \<\n_\th \mathcal{T}_0(u), u_\th\>\right| \le C \alpha e^{2s} \left( \int_{S^1} |\n_\th \mathcal{T}_0(u)|^2 \right)^{1/2} f.
\end{equation}
By (\ref{e:small-assumption}), we may estimate
\begin{equation*} 
\begin{split}
|\nabla_\theta \T_0| = |\nabla_\theta \left( \zeta^2 e^{2s} \T \right) |
& \leq e^{2s} \left| \p_\theta \right|_g \left( |2 \zeta d\zeta |_g |\T|_g + \zeta^2 |\nabla \T |_g \right) \\
& \leq C e^{3s} \left( \alpha e^s |\nabla du|_g + |\nabla^2 du|_g \right) \\
& \leq C \eta.
\end{split}
\end{equation*}
Then (\ref{e:4-2}) gives
\begin{equation*}
  |Q(u)|\le  C \alpha \eta e^{2s} f.
\end{equation*}
The bound (\ref{e:small-assumption}) also gives $|u_s|\le \eta$. Hence
\begin{equation*}
\left|\int_{S^1}\<R(u_s, u_\th)u_s,u_{\th}\>\right|\le C_N \eta^2 f^2 \leq \eta f^2
\end{equation*}
for $\eta$ sufficiently small. Overall, (\ref{e:4-1}) now reads
\begin{equation}\label{e:4-3.5} 
\frac12(e^{2s}\p_t-\p_s^2) f^2 \leq - \int_{S^1} \left( |\n_s u_{\th}|^2+|\n_{\th} u_{\th}|^2 \right) + \eta f^2 + C \alpha  \eta e^{2s} f.
\end{equation}
Since $|u_\th|\le \eta$ is small, we may also assume that the image of the curve $u(s, \cdot, t):S^1\to N$ lies in a coordinate chart of $N$ where the Christoffel symbol $\Ga$ is bounded by $C_N$. Then in local coordinates, we have
\[ |\p_\th u_\th| = |\n_\th u_\th - \Ga(u_\th, u_\th)|\le |\n_\th u_\th| + |\Ga(u_\th, u_\th)| \le \eta + C_N \eta^2 \leq 2 \eta.\]
Thus
\[\begin{aligned}
  |\n_{\th} u_{\th}|^2&=|\p_\th u_\th +\Ga(u_\th, u_\th)|^2\\
  &\ge   |\p_\th u_\th|^2-2|\p_\th u_\th||\Ga(u_\th, u_\th)|+|\Ga(u_\th, u_\th)|^2\\
  &\ge  |\p_\th u_\th|^2-C \eta |u_\th|^2.
\end{aligned}\]
Then the classical Poincar\'e inequality on $S^1$ yields
\begin{equation}\label{e:4-4}
  \int_{S^1}|\n_{\th} u_{\th}|^2\ge  \int_{S^2}|u_\th|^2-C \eta \int_{S^1}|u_\th|^2=(1-C\eta )f^2.
\end{equation}
Combining 
(\ref{e:4-3.5}) and (\ref{e:4-4}), we arrive at
\begin{equation}\label{e:angular}
  \frac12(e^{2s}\p_t-\p_s^2)f^2\le -\int_{S^1}|\n_s u_{\th}|^2-(1- (C+1) \eta) f^2+C\alpha \eta e^{2s} f.
\end{equation}
Next, note that
\[ \frac12\p_s f^2=f\p_sf=\int_{S^1}\<\n_s u_{\th},u_\th\>\le \(\int_{S^1}|\n_s u_{\th}|^2\)^{1/2}f,\]
so we have $|\p_sf|^2\le\int |\n_s u_{\th}|^2$. Inserting into (\ref{e:angular}), we get
\begin{equation}\label{e:4-5}
  \frac12(e^{2s}\p_t-\p_s^2)f^2\le -|\p_sf|^2-(1-C \eta )f^2+C \alpha \eta e^{2s} f.
\end{equation}
On the other hand,
\[ \frac12\p_s^2(f^2)=f\cdot \p_s^2f+|\p_s f|^2.\]
Inserting this into (\ref{e:4-5}), we obtain
\[ e^{2s}\p_tf-\p_s^2f+(1-C \eta )f\le C  \alpha \eta e^{2s}.\]
Finally, translating the above equation back to polar coordinates and dividing by $r^2,$ 
we get (\ref{e:angular-energy}) and the lemma is proved.
\end{proof}


\vspace{5mm}

\subsection{Construction of supersolution}\label{ss:constructionofsuper} Given $\frac{1}{2} \leq \gamma < 1$ and $T > 1,$
define the spacetime region
\begin{equation}
\Omega_\gamma = \left\{ (r,t) \in \LB R, 1 \RB \times \LB 0, T \right) \mid 1 - t \leq r^{2\gamma} \leq 1 \right\}.
\end{equation}
\begin{center}
\includegraphics[scale=.4]{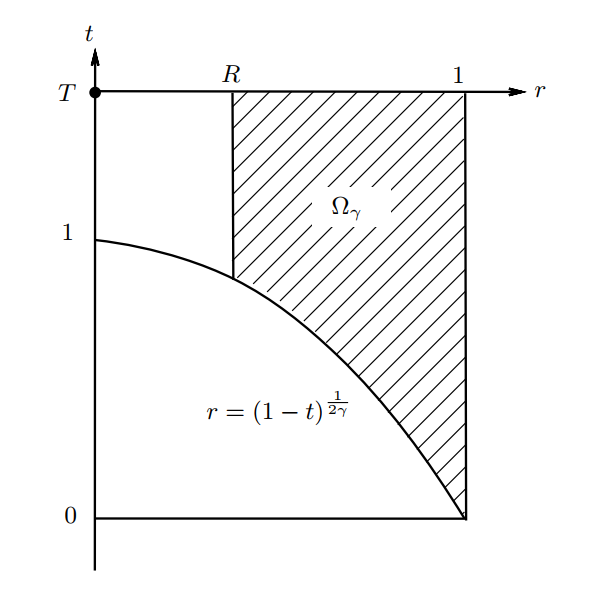}
\end{center}

We will now construct a supersolution for the equation (\ref{e:angular-energy}) on $\Omega_\gamma,$ with controlled boundary values.

Letting
\begin{equation}\label{lambdagammadef}
\ld^{R,1}_\ga (t):= \max \LB R, (1-t)_+^{\frac{1}{2\ga}} \RB,
\end{equation}
the parabolic boundary of $\Omega_\gamma$ is given by
\begin{equation*}
\begin{split}
\partial \Omega_\gamma & = \{1\} \times \LB 0 , T \right) \cup \{(\lambda_\gamma^{R,1}(t), t) \mid t \in \LB 0, T \right) \}. \\
\end{split}
\end{equation*}
Let $\frac{1}{2} \leq \gamma < \nu \leq 1$
and consider the operator
\begin{equation}
\Delta_\nu = \pd_r^2 + \frac{1}{r} \pd_r - \frac{\nu^2}{r^2}.
\end{equation}
Let
$$v_0(r,t) = \frac{\left( (1 - t)_+ + r^{2\nu} \right)^{\frac{\nu}{2 \gamma}}}{r^{\nu}}.$$
Choose $\mu$ with
\begin{equation}\label{muinterval}
0 < \mu < \min \LB \nu, \frac{\nu^2}{\gamma} - 3 \nu + 2 \RB.
\end{equation}
Now let
\begin{equation}\label{super-sol}
v(r,t) = v_0(r,t) + \left( \frac{R}{r} \right)^\nu + \frac{\nu + 1 }{ \nu^2 - \mu^2 } r^{\mu}.
\end{equation}
\begin{lemma}\label{l:super-sol}
On $\Omega_\gamma,$ the function $v(r,t)$ satisfies
\begin{equation}\label{super-sol:supersol}
\left( \partial_t - \Delta_\nu \right) v \geq r^{\mu - 2}
\end{equation}
with
\begin{equation}\label{super-sol:boundarybound}
1 \leq \sup_{\partial \Omega_\gamma} v \leq C_{\mu, \nu}
\end{equation}
and
\begin{equation}\label{super-sol:vbound}
v(r,t) \leq C_{\mu, \nu} \left( \left( \frac{\lambda^{R,1}_\gamma(t)}{r} \right)^\nu + r^{\min \LB \mu, \nu \left( \frac{\nu}{\gamma} - 1 \right) \RB} \right).
\end{equation}
\end{lemma}
\begin{proof} We first claim that
\begin{equation}\label{ellipticsupsol}
\Delta_\nu v_0 \leq 0.
\end{equation}
By a straightforward calculation, one finds that
\begin{equation*}
\begin{split}
\partial_r v_0 
& = \nu v_0 \LB \frac{\nu}{\gamma} \frac{r^{2 \nu - 1}}{(1 - t + r^{2 \nu})}  - \frac{1}{r} \RB
\end{split}
\end{equation*}
and
\begin{equation*}
\begin{split}
\partial_r^2 v_0 
& = \nu v_0 \LB \frac{\nu^2}{\gamma^2} \frac{r^{4 \nu -2}}{(1 - t + r^{2 \nu})^2}  \left(\nu - 2 \gamma \right) - \frac{\nu}{\gamma} \frac{r^{2 \nu - 2}}{1 - t + r^{2 \nu} } + \frac{\nu + 1}{r^2} \RB.
\end{split}
\end{equation*}
This gives
\begin{equation*}
\begin{split}
\left( \pd_r^2 + \frac{1}{r} \pd_r - \frac{\nu^2}{r^2} \right) v_0 
& = \frac{\nu^3}{\gamma^2} v_0 \frac{r^{4 \nu - 2}}{(1 - t + r^{2 \nu})^2 } \left( \nu - 2 \gamma \right),
\end{split}
\end{equation*}
which is non-positive for $t \leq 1,$ since $\nu - 2 \gamma \leq \nu - 1 \leq 0.$ This proves (\ref{ellipticsupsol}).

Because $2 \gamma \geq \nu,$ we also have
\begin{equation*}
\begin{split}
\partial_t v = \partial_t v_0 = \frac{-\left( \frac{\nu}{2 \gamma}\right) }{ \left(1 - t + r^{2 \nu}\right)^{\frac{- \nu + 2 \gamma}{2 \gamma} } r^{\nu} } \geq  -\left( \frac{\nu}{2 \gamma}\right) r^{-\frac{ \nu(2 \gamma - \nu) + \nu \gamma }{\gamma}} \geq - \nu r^{-\frac{ \nu( 3 \gamma - \nu) }{\gamma}}.
\end{split}
\end{equation*}
Note that
$$\Delta_\nu r^{\mu} = \left( \mu^2 - \nu^2 \right) r^{\mu - 2} .$$
By (\ref{muinterval}), we have
\begin{equation*}
\mu - 2 \leq \frac{\nu}{\gamma} \left( \nu - 3 \gamma \right).
\end{equation*}
We therefore obtain
\begin{equation}\label{e:5-4}
\left( \partial_t - \Delta_\nu \right) v \geq - \nu r^{-\frac{ \nu( 3 \gamma - \nu) }{\gamma}} + (\nu+1) r^{\mu - 2} \geq r^{\mu - 2}
\end{equation}
for $r \leq 1$, which establishes (\ref{super-sol:supersol}).

The bounds (\ref{super-sol:boundarybound}-\ref{super-sol:vbound}) follow from (\ref{super-sol}) by inspection. 
\end{proof}

\begin{prop}\label{p:decay-est}
Let $\frac12 \leq \gamma < \nu \leq 1,$ and $\mu$ be as in (\ref{muinterval}). Suppose $g(r,t):\Om_\ga \to \R$ satisfies the equation
\begin{equation*}\label{e:5-0}
(\p_t-\De_\nu) g \le A r^{\mu - 2}
\end{equation*}
with
\begin{equation*}
\sup_{\partial \Omega_\gamma} g \leq A.
\end{equation*}
Then $g$ satisfies a bound
\begin{equation}\label{e:decay-est}
g (r,t)\le C_{\mu, \nu} A \left( \left( \frac{\lambda^{R,1}_\gamma(t)}{r} \right)^\nu + r^{\min \LB \mu, \nu \left( \frac{\nu}{\gamma} - 1 \right) \RB} \right)
\end{equation}
for $(r,t) \in \Omega_\gamma.$
\end{prop}
\begin{proof} This follows from the previous lemma by the comparison principle, applied to the functions $g$ and $A v(r,t).$
\end{proof}

\vspace{5mm}

\subsection{Decay estimates}\label{ss:decayests}

We can now prove a spatial decay estimate for the angular component 
of the energy (Lemma \ref{lemma:angulardecay}). This will be combined with our assumption on the stress-energy tensor (or on $\dbar u$) to obtain bounds on $du$ 
in the neck region (Theorems \ref{thm:decayest}-\ref{thm:dbaruboundeddecayest}).

\begin{lemma}\label{lemma:angulardecay} Let $0 < \eta^2 < \varepsilon_0,$ depending on the geometry of $N,$ and $\frac{1}{2} \leq \gamma < 1.$ Let $u$ be a solution of (\ref{e:hm}) on $U_R^1 \times \LB 0, T \right),$ with $T > 1,$ where $U_R^1$ has a metric $g$ as in (\ref{e:4-metric}). Suppose that for all $(x,t)$ with
\begin{equation}\label{angulardecay:range}
\max \LB R^{2 \gamma}, (1-t) \RB \leq |x|^{2 \gamma} \leq 1
\end{equation}
we have
\begin{equation}
|x| |du (x,t)| + |x|^2|\nabla du(x,t)| + |x|^3 |\nabla^{(2)} du(x,t) | \leq \eta.
\end{equation}
Then for
\begin{equation}\label{angulardecay:nurange}
\gamma < \nu \leq \sqrt{1 - C \eta}
\end{equation}
and for all $(r,t)$ as in (\ref{angulardecay:range}), the angular energy (\ref{angularenergy}) satisfies
\begin{equation}
f(u; r,t) \leq C_{\gamma, \nu} \eta \left( \left( \frac{\lambda^{R,1}_\gamma(t) }{r}\right)^{\nu} + r^{\nu \left( \frac{\nu}{\gamma} - 1 \right)} \right).
\end{equation}
\end{lemma}
\begin{proof} This follows from Lemma \ref{l:equ-angular} and Proposition \ref{p:decay-est}, where we let $\mu = \nu \left( \frac{\nu}{\gamma} - 1 \right)$ and $A = C \eta.$
\end{proof}

Given $\frac12 \leq \gamma < 1,$ we let
\begin{equation}\label{lambdagammaRrhodef}
\lambda^{R, \rho}_\gamma(t) = \max \LB R, \rho \left( 1 - \frac{t}{\rho^2} \right)^{1/2\gamma} \RB.
\end{equation}



\begin{thm}\label{thm:decayest} There exists $\varepsilon_0 > 0,$ depending on the geometry of $N,$ and $R_0 > 0,$ depending on the geometry of $\Sigma,$ as follows. 

For $0 < R \leq \rho < \min \LB R_0, \sqrt{T} \RB$ and $p \in \Sigma,$ let $u : U^{2\rho}_{R/2}(p) \times \LB 0, T \right) \to N$ be a solution of (\ref{e:hm}) that satisfies (\ref{stressbound}) for some $q > 1.$
Suppose that $\tau + \rho^2 \leq t_0 < T$ is such that for all $r,t$ with $\lambda^{R, \rho}_\gamma (t - t_0 + \rho^2) \leq r \leq 2\rho,$ we have
\begin{equation}\label{decayest:Ebound}
E( u(t), U^r_{r/2} ) \leq \varepsilon < \varepsilon_0.
\end{equation}
Suppose that $\nu$ satisfies 
\begin{equation}\label{decayest:nurange}
\gamma < \nu \leq \sqrt{1 - C \sqrt{\eps} }.
\end{equation}
Then for all $R \leq r \leq \rho,$ we have
\begin{equation}\label{decayest:est}
r | du(x,t_0)| \leq C_{\gamma, \nu} \sqrt{\varepsilon} \left( \left( \frac{R}{r}\right)^{\nu} + \left( \frac{r}{\rho} \right)^{\nu \left( \frac{\nu}{\gamma} - 1 \right)} \right) + C\sqrt{\sigma} r^{ 1 - \frac{1}{q} },
\end{equation}
where $r = \dist_g(x,p).$
\end{thm}
\begin{proof} For $R < R_0$ sufficiently small, radial balls in conformal and geodesic coordinates are uniformly equivalent, so we may conflate the two in the statement.

We rescale $u$ by
\begin{equation}
u'(x', t') = u\left( \rho x', t_0 + \rho^2 (t' - 1) \right).
\end{equation}
Then $u'$ solves (\ref{e:hm}) on an annulus $U = U_{R/\rho}^1,$ with $\alpha$ arbitrarily small in (\ref{e:5-metric}). 
By the $\varepsilon$-Regularity Theorem \ref{thm:epsilonreg}, applied on 
parabolic cylinders covering the domain, 
(\ref{decayest:Ebound}) implies a bound
\begin{equation*}
\sup_{\lambda^{R/\rho,1}_\gamma(t') \leq |x'| \leq 1} \left( r' |d u' (x',t')| + r'^2|\nabla du' (x',t')| + r'^3 |\nabla^{(2)} d u'(x',t') | \right) \leq C \sqrt{\varepsilon}.
\end{equation*}
This is equivalent to (\ref{e:small-assumption}), with $\eta = C \sqrt{\varepsilon},$ and (\ref{decayest:nurange}) implies (\ref{angulardecay:nurange}). We may therefore apply Lemma \ref{lemma:angulardecay}, to obtain a bound
\begin{equation*}
f(u';r',t') \leq C_{\nu, \gamma} \sqrt{\varepsilon} \left( \left( \frac{R}{\rho r' } \right)^{\nu} + r'^{\nu \left( \frac{\nu}{\gamma} - 1 \right)} \right).
\end{equation*}
Undoing the rescaling, we obtain
\begin{equation}\label{decayest:radialest}
f(u; r,t) \leq  C_{\nu, \gamma} \sqrt{\varepsilon} \left( \left( \frac{R}{r} \right)^{\nu} + \left( \frac{r}{\rho} \right)^{\nu \left( \frac{\nu}{\gamma} - 1 \right) } \right)
\end{equation}
for $\lambda^{R, \rho}_\gamma (t - t_0 + \rho^2) \leq r \leq \rho.$

It remains to bound the radial energy using the stress-energy bound (\ref{stressbound}).
By definition (\ref{stressenergy}), the stress-energy tensor is conformally invariant and has the form
\[ S=\frac12\(|u_r|^2-\frac{1}{r^2}|u_\th|^2\)(dr^2-r^2d\th^2)+2\<u_r, u_\th\>drd\th.\]
It follows that
\begin{equation}\label{decayest:du2est} |du|^2=|u_r|^2+ \frac{1}{r^2}|u_\th|^2\le \frac{2}{r^2}|u_\th|^2+\sqrt{2} |S|.
\end{equation}
Therefore, by H\"older's inequality and (\ref{stressbound}), we have
\begin{equation*}
  \begin{aligned}
  E(u(t), U_{r/2}^r)& \le \int_{U_{r/2}^r}\frac{2}{r^2}|u_\th|^2+ \sqrt{2} \int_{U_{r/2}^r}|S|\\
  & \leq \int_{r/2}^r\frac{2}{r}f^2(r,t) dr+C\sigma r^{2\(1-\frac{1}{q}\)}.
\end{aligned}
\end{equation*}
This combined with (\ref{decayest:radialest}) gives
\[  E(u(t), U_{r/2}^r)\le C_{\nu, \gamma} \varepsilon \left( \left( \frac{R}{r}\right)^{2\nu} + \left( \frac{r}{\rho} \right)^{2\nu \left( \frac{\nu}{\gamma} - 1 \right)} \right) + C \sigma r^{ 2\(1 - \frac{1}{q} \)}.\]
Now the desired estimate (\ref{decayest:est}) follows again from the $\ep$-regularity Theorem~\ref{thm:epsilonreg}, applied on parabolic cylinders. 
\end{proof}

\begin{thm}\label{thm:dbaruboundeddecayest} Let $u$ be as in the previous theorem, and assume that $N$ is a Hermitian manifold. In place of (\ref{stressbound}), assume that
$$\sup_{U^{2\rho}_{R/2} \times \LB \tau + \frac{\rho^2}{2} , T \right) } e_{\dbar}(u) \leq \delta.$$
We then have
\begin{equation*}
r | du(x,t_0)| \leq C_{\gamma, \nu} \sqrt{\varepsilon} \left( \left( \frac{R}{r}\right)^{\nu} + \left( \frac{r}{\rho} \right)^{\nu \left( \frac{\nu}{\gamma} - 1 \right)} \right) + C \sqrt{\delta} r.
\end{equation*}
\end{thm}
\begin{proof} We argue exactly as in the previous theorem, up to (\ref{decayest:du2est}). We have $\bp u = \left( u_r + \frac{i}{r} u_\theta \right) d\bar{z},$ hence
$$|u_r|^2 \leq 2 \left( \frac{1}{r^2} |u_\theta|^2 + e_\dbar \right)$$
and
\begin{equation*}
|du|^2 \leq \frac{3}{r^2} |u_\theta|^2 + 2 e_\dbar \leq \frac{3}{r^2} |u_\theta|^2 + 2 \delta.
\end{equation*}
The result follows by using this identity in place of (\ref{decayest:du2est}).
\end{proof}

\vspace{5mm}

\section{H\"older continuity and bubble-tree convergence}\label{sec:maintheorems}

We now assemble our main technical results, Theorems \ref{thm:maintechnical} and \ref{thm:lowdbarmaintechnical}; 
these give strong decay estimates in the neck region, assuming only the stress-energy bound (\ref{stressbound}) or non-concentration of the (anti-)holomorphic energy. 
H{\"o}lder continuity of the body map follows directly from these estimates, as do the energy identity and no-neck properties in the bubble-tree decomposition.

\subsection{H{\"o}lder continuity assuming stress-energy bound}\label{ss:holderstress} 

Throughout \S \ref{ss:holderstress}-\ref{ss:holdernonconc}, $u$ will denote a solution of (\ref{e:hm}) on $\Sigma \times \LB 0, T \right),$ which we assume to be smooth for $0 < t < T.$

\begin{thm}\label{thm:maintechnical} Let $K, q > 1,$ and $\nu > 0$ with
\begin{equation}\label{maintechnical:qsize}
\max \LB \frac{1}{2}, \frac{1}{q} \RB \cdot \frac{2q - 1}{q} \leq \nu^2 < 1. 
\end{equation}
There exists $\varepsilon_1 > 0,$ depending on $K, q, \nu,$ and the geometry of $N,$ as well as $R_0 > 0,$ depending on the geometry of $\Sigma,$ as follows.

Let $\sigma > 0,$ $0 < \varepsilon < \varepsilon_1,$ and
\begin{equation}\label{maintechnical:rhosmallness}
0 < \rho < \min \LB R_0, \sqrt{T}, \left( \frac{\varepsilon}{C \sigma} \right)^{\frac{q}{2q-2}} \RB.
\end{equation}
Let $p \in \Sigma,$ and suppose that for some $0 \leq \tau < T - \rho^2,$ the stress-energy bound (\ref{stressbound}) holds on $B_\rho = B_\rho(p).$ Suppose further that $t \in \LB \tau + \rho^2, T \right)$ is such
that 
\begin{equation}\label{maintechnical:assumption}
\sup_{t - \rho^2 \leq s \leq t} E(u(s), B_\rho ) \leq E(u(t), B_{\rho / 2} ) + K\varepsilon.
\end{equation}
Then for $2\lambda_{\varepsilon, \rho , p}(u(t)) \leq r \leq \rho/2,$ 
we have
\begin{equation}\label{maintechnical:estimate}
r| du(x,t)| \leq C_{K,q,\nu} \sqrt{\varepsilon} \left( \left( \frac{\lambda_{\varepsilon, \rho , p}(u(t)) }{r} \right)^\nu + \left( \frac{r}{\rho} \right)^{ 1 - \frac{1}{q} } \right).
\end{equation}
Here $r = \dist(x,p).$
\end{thm}
\begin{proof} Let $\gamma = \max \LB \frac12, \frac1q \RB.$ Then (\ref{maintechnical:qsize}) reads
$\gamma\left( 2 - \frac{1}{q} \right) \leq \nu^2,$
which implies
\begin{equation}\label{maintechnical:qunugammabound}
1 - \frac{1}{q} < 2 - \frac{1}{q} - \nu \leq \nu \left( \frac{\nu}{\gamma} - 1 \right).
\end{equation}
Let $n \geq K +1,$ to be chosen below depending on $K, q,$ and $\nu.$ We may assume that $\varepsilon_1 > 0$ is small enough that
$3 n \varepsilon_1 < \varepsilon_0,$ the constant of Theorem \ref{thm:decayest}, and $\sqrt{1 - C\sqrt{n\varepsilon_1}} \geq \nu,$ in order to satisfy (\ref{decayest:nurange}). 

Now, let $0 \leq R < \rho / 2$ be the minimal number such that
\begin{equation}\label{maintechnical:Rdef}
\sup_{t - \rho^2 \leq s \leq t} E(u(s), B_\rho ) \leq E(u(t), B_{R}) + 3 n \varepsilon.
\end{equation}
By Theorem \ref{thm:lambdaest}, we have
\begin{equation}\label{maintechnical:2}
\lambda_{3n \varepsilon, \rho, p} (u(s)) \leq \max \LB 4 R, \left( \frac{ C\sigma (t - s) }{ \varepsilon } \right)^{\frac{q}{2}} \RB
\end{equation}
for $t - \rho^2 \leq s \leq t.$ But, by (\ref{maintechnical:rhosmallness}), we have
\begin{equation*}
 \left( \frac{ C\sigma }{ \varepsilon } \right)^{\frac{q}{2}} \leq \rho^{1 - q},
\end{equation*}
and so (\ref{maintechnical:2}) implies
\begin{equation}\label{maintechincal:3}
\begin{split}
\lambda_{3n \varepsilon, \rho, p} (u(s)) & \leq \max \LB 4 R, \rho^{1 - q} \left( t - s \right)^{\frac{q}{2}} \RB \\
& = \max \LB 4 R, \rho \left( \frac{t - s}{\rho^2} \right)^{ \frac{q}{2} } \RB \\
& \leq \lambda_{\gamma}^{4R, \rho} (s - t + \rho^2),
\end{split}
\end{equation}
since $q/2 \geq 1 / 2\gamma.$ Here $\lambda_\gamma^{R, \rho}(\cdot)$ is defined by (\ref{lambdagammaRrhodef}).
Therefore (\ref{decayest:Ebound}) is satisfied, and we conclude from Theorem \ref{thm:decayest} that
\begin{equation}\label{maintechnical:4}
\begin{split}
r|du(t)| & \leq C_{q,\nu} \sqrt{n \varepsilon} \left( \left( \frac{4R}{r}\right)^{\nu} + \left( \frac{r}{\rho} \right)^{\nu \left( \frac{\nu}{\gamma} - 1 \right)} \right) + C \sqrt{ \sigma} r^{ 1 - \frac{1}{q} } \\
& \leq C_{q, \nu} \sqrt{n \varepsilon} \left( \left( \frac{R}{r}\right)^{\nu} + \left( \frac{r}{\rho} \right)^{1 - \frac1q} \right) + C \sqrt{ \sigma} r^{ 1 - \frac{1}{q} }
\end{split}
\end{equation}
for $4R \leq r \leq \rho / 2,$ where we have applied (\ref{maintechnical:qunugammabound}). But, again by (\ref{maintechnical:rhosmallness}), we have
$$\sqrt{\sigma} \leq \sqrt{\varepsilon} \rho^{-\frac{q-1}{q}}$$
and
\begin{equation}
\sqrt{\sigma} r^{1 - \frac1q} \leq \sqrt{\varepsilon} \left( \frac{r}{\rho} \right)^{1 - \frac1q}.
\end{equation}
Hence (\ref{maintechnical:4}) simplifies to
\begin{equation}\label{maintechnical:5}
\begin{split}
r|du(t)| & \leq C_{q,\nu} \sqrt{n \varepsilon} \left( \left( \frac{R}{r}\right)^{\nu} + \left( \frac{r}{\rho} \right)^{1 - \frac1q} \right).
\end{split}
\end{equation}
Finally, we claim that $R \leq \lambda_{\varepsilon, \rho , p}(u(t)) / 2$ for the appropriate choice of $n.$ Assuming the contrary, we must have $R > 0,$ $\lambda = \lambda_{\varepsilon, \rho , p}(u(t)) \leq \rho / 2,$ and equality in (\ref{maintechnical:Rdef}):
\begin{equation*}
\sup_{t - \rho^2 \leq s \leq t} E(u(s), B_\rho ) = E \left( u(t), B_{R} \right) + 3 n \varepsilon.
\end{equation*}
Then from (\ref{maintechnical:assumption}), we have
\begin{equation*}
E(u(t), B_{R}) + 3 n \varepsilon \leq E \left( u(t), B_{\rho / 2} \right) + K \varepsilon
\end{equation*}
and
\begin{equation}\label{maintechnical:tobecontradicted}
 (3 n - K) \varepsilon \leq E \left( u(t), U_R^{\rho / 2} \right).
\end{equation}
But, since $\lambda < 2R,$ we have
\begin{equation}\label{maintechnical:6}
E(u(t), U_R^{2^nR} \cup U_{2^{-n} \rho }^{\rho}) < 2 n \varepsilon.
\end{equation}
By (\ref{maintechnical:5}), we also have
\begin{equation*}
\begin{split}
E \left( u(t), U_{2^nR}^{2^{-n} \rho} \right) & \leq C_{q, \nu} n \varepsilon \int_{2^n R}^{2^{-n} \rho} \left( \left( \frac{R}{r}\right)^{\nu} + \left( \frac{r}{\rho} \right)^{1 - \frac1q} \right)^2 \, \frac{dr}{r} \\
& \leq C_{q, \nu} n \varepsilon 2^{- \left( 1 - \frac1q \right) n}.
\end{split}
\end{equation*}
Choosing $n$ sufficiently large (depending on $q$ and $\nu$), we obtain
\begin{equation}\label{maintechnical:7}
E \left( u(t), U_{2^n R}^{2^{-n} \rho}  \right) \leq \varepsilon.
\end{equation}
Combining (\ref{maintechnical:tobecontradicted}), (\ref{maintechnical:6}), and (\ref{maintechnical:7}), we have
\begin{equation*}
\begin{split}
(3n - K) \varepsilon & \leq E \left( u(t), U_R^{\rho/2} \right) \\
& = E(u(t), U_R^{2^nR} \cup U_{2^{-n} \rho }^{\rho}) + E \left( u(t), U_{2^n R}^{2^{-n} \rho}  \right) \\
& < 2n \varepsilon + \varepsilon = (2n + 1) \varepsilon.
\end{split}
\end{equation*}
Since $n \geq K + 1,$ this is a contradiction.

We have established that $R \leq \lambda_{\varepsilon, \rho, p} (u(t)) / 2,$ and the desired estimate (\ref{maintechnical:estimate}) now follows from (\ref{maintechnical:5}). 
\end{proof}

\begin{cor}\label{cor:delta0} Given $q > 1,$ there exists $\eps_2 > 0,$ depending on $q$ and the geometry of $N,$ as well as $R_0 > 0,$ depending on the geometry of $\Sigma,$ as follows.

Let $\sigma > 0,$ $0 < \eps < \eps_2,$ and $\rho>0$ satisfying (\ref{maintechnical:rhosmallness}). 
Suppose that (\ref{stressbound}) is satisfied on $B_\rho,$ together with 
\begin{equation}\label{delta0:assumption}
\int_\tau^T \!\!\!\!\! \int_{B_\rho} | \mathcal{T}(u(s)) |^2 \, dV ds \leq \frac{c \varepsilon^2 }{E(u(\tau))}.
\end{equation}
Then for $\tau + \rho^2 \leq t < T$ and $2\lambda_{\varepsilon, \rho , p}(u(t)) \leq r \leq \rho/2,$ we have 
\begin{equation*}
r|du(x,t)| \leq C_{q} \sqrt{\varepsilon} \left(\frac{\lambda_{\varepsilon, \rho , p}(u(t)) }{r}  + \frac{r}{\rho} \right)^{ 1 - \frac{1}{q}}.
\end{equation*}
\end{cor}
\begin{proof} Fixing $K = 3$ and $\nu = \sqrt{ \max \LB \frac{1}{2}, \frac{1}{q} \RB \cdot \frac{2q - 1}{q} }  > 1 - \frac1q,$ we may take $\eps_2 = \eps_1$ from the previous theorem.
It suffices to show that (\ref{delta0:assumption}) implies (\ref{maintechnical:assumption}), which can be done by the following standard argument.

Fix $t \in \LB \tau + \rho^2, T \right).$ We may assume without loss of generality that $\lambda_{\varepsilon, \rho , p} (u(t)) \leq \rho/4,$ otherwise the claim is vacuous. 
This implies
\begin{equation}\label{lambdarho/2}
E \left( u(t), U^\rho_{\rho/4} \right) < 2\varepsilon.
\end{equation}
Let $\varphi$ be a cutoff for $B_{\rho/2} \subset B_{\rho}.$ Integrating once against $\varphi$ in (\ref{pointwiseenergy}), and inserting (\ref{divS}), we obtain
\begin{equation*}
\frac12 \frac{d}{dt} \int \varphi | \nabla u|^2 \, dV + \int \varphi |\mathcal{T}(u)|^2 \, dV = \int \LA \mathcal{T}(u) , \nabla \varphi \cdot \nabla u \RA.
\end{equation*}
Given $t - \rho^2 \leq s \leq t,$ integrating from $s$ to $t$ and using H\"older's inequality, we obtain 
\begin{equation}\label{holdercrap}
\begin{split}
\left| \int \varphi \left(| \nabla u(s)|^2 - |\nabla u(t) |^2 \right) \, dV \right| & \leq \frac{c \varepsilon^2 }{E(u(\tau))} + \frac{C}{\rho} \frac{\sqrt{c}\, \varepsilon}{\sqrt{E(u(\tau))}} \sqrt{\rho^2E(u(\tau))} \\
& \leq \frac{c \varepsilon^2}{E(u(\tau))} + C \sqrt{c} \varepsilon.
\end{split}
\end{equation}
We may assume that $E(u(\tau)) \geq \varepsilon$ without loss of generality (otherwise the estimate follows from Theorem \ref{thm:epsilonreg}), 
and $c$ is sufficiently small. Inserting (\ref{lambdarho/2}), 
(\ref{holdercrap}) becomes
\begin{equation*}
\begin{split}
\int_{B_{\rho / 2}} |\nabla u(s) |^2  \, dV \leq \int \varphi |\nabla u(s) |^2  \, dV  & \leq \int \varphi | \nabla u(t)|^2 \, dV + \varepsilon \leq \int_{B_{\rho / 4}} | \nabla u(t)|^2 \, dV + 3 \varepsilon.
\end{split}
\end{equation*}
This gives (\ref{maintechnical:assumption}), with $K = 3$ and $\rho / 2$ in place of $\rho,$ which is sufficient.
\end{proof}

\begin{cor}\label{cor:stressboundedholdercont} 
If $u$ satisfies (\ref{stressbound}) for $T < \infty,$ with $q > 1,$ then the body map $u(T) = \lim_{t \nearrow T} u(t)$ is $C^{1 - \frac{1}{q}}.$
\end{cor}
\begin{proof} Let $\eps = \eps_2 /2,$ and fix any point $p \in \Sigma.$ Since $E(u(t))$ is is continuous and decreasing, we may choose $\tau < T$ such that
$$E(u(\tau)) - \lim_{s \nearrow T} E(u(s)) \leq \frac{c \varepsilon^2}{E(u(\tau))}.$$
By the global energy identity (\ref{globalenergyidentity}), this implies
$$\int_\tau^T \!\!\!\!\! \int_{\Sigma} | \mathcal{T}(u(s)) |^2 \, dV ds \leq \frac{c \varepsilon^2 }{E(u(\tau))},$$
which guarantees (\ref{delta0:assumption}).

Now, by Corollary \ref{cor:lambdatozero}, we may choose choose $0 < \rho < \sqrt{T - \tau}$ such that
$$\lambda_{\eps, \rho, p}(u(t)) \to 0$$
as $t \nearrow T;$ in particular, $B_\rho(p) \setminus \{p\}$ does not contain any singular points. Taking $\rho$ still smaller, we may assume (\ref{maintechnical:rhosmallness}).

We can now apply the prevous corollary, to obtain
\begin{equation*}
r| du(x,t)| \leq C_{q,N} \left(\frac{\lambda_{\varepsilon, \rho , p}(u(t)) }{r}  + \frac{r}{\rho} \right)^{ 1 - \frac{1}{q}}.
\end{equation*}
Taking the limit $t \nearrow T$ on $B_\rho(p) \setminus \{ p\},$ we have
\begin{equation}\label{holderwithstressbound}
r| du(x,T)| \leq C_{q,N} \left( \frac{r}{ \rho} \right)^{1 - \frac{1}{q} } .
\end{equation}
H{\"o}lder continuity now follows from (\ref{holderwithstressbound}).
\end{proof}

\vspace{5mm}

\subsection{H{\"o}lder continuity assuming non-concentration of (anti-)holomorphic energy}\label{ss:holdernonconc}

We now repeat the arguments of the previous subsection under our hypotheses in the compact K{\"a}hler case, obtaining H{\"o}lder continuity for all exponents less than one.

\begin{thm}\label{thm:lowdbarmaintechnical} Suppose that $N$ is compact K\"ahler with nonnegative holomorphic bisectional curvature. Given $0 < \nu < 1,$
there exists $\varepsilon_3 > 0,$ depending on $\nu$ and the geometry of $N,$ as well as $R_0 > 0,$ depending on the geometry of $\Sigma,$ as follows.

Let $0 < \varepsilon < \varepsilon_3,$ $0 < \delta < \delta_0,$ $0 \leq \tau < T,$ and $0 < 2 \rho \leq\rho_1 < \min \LB R_0, \sqrt{T - \tau} \RB$ with
\begin{equation}\label{lowdbarmaintechnical:rhosmallness}
\rho < \frac{\varepsilon \rho_1}{C \sqrt{\delta E(u(\tau)) }}.
\end{equation}
Suppose that
\begin{equation}\label{lowdbarmaintechnical:assumption1}
\sup_{\tau \leq t < T} E_{\dbar} (u(t), B_{\rho_1} ) \leq \delta < \delta_0
\end{equation}
and
\begin{equation}\label{lowdbarmaintechnical:assumption2}
\int_\tau^T \!\!\!\!\! \int_{B_{\rho_1}} | \mathcal{T}(u(s)) |^2 \, dV ds \leq \frac{c \eps^2}{E(\tau)}.
\end{equation}
Then for $2\lambda_{\varepsilon, \rho , p}(t) \leq r \leq \rho/2,$ 
$u(t)$ satisfies
\begin{equation}\label{lowdbarmaintechnical:estimate}
r|du(x,t)| \leq C_{\nu} \sqrt{\varepsilon} \left( \left( \frac{\lambda_{\varepsilon, \rho , p}(u(t)) }{r} \right)^\nu + \left( \frac{r}{\rho} \right)^{\nu(2\nu - 1) } \right).
\end{equation}
\end{thm}

\begin{proof} We may assume without loss of generality that $\nu \geq \frac{\sqrt{3}}{2} ,$ and let $q = 2,$ $\gamma = \frac12$ in the previous arguments.
By the proof of Corollary \ref{cor:delta0}, we know that (\ref{lowdbarmaintechnical:assumption2}) implies (\ref{maintechnical:assumption}).
By Corollary \ref{cor:nnhbscblowuprate}, we may let $\sigma = \frac{C \sqrt{\delta E(u(\tau))} }{\rho_1},$ so that (\ref{lowdbarmaintechnical:rhosmallness}) implies (\ref{maintechnical:rhosmallness}).

We can now rerun the proof of Theorem \ref{thm:maintechnical}, using Theorems \ref{thm:e''bound} and \ref{thm:dbaruboundeddecayest} in place of Theorem \ref{thm:decayest} after (\ref{maintechincal:3}). Instead of (\ref{maintechnical:4}), we obtain
\begin{equation}\label{lowdbarmaintechnical:4}
\begin{split}
r|du(t)| & \leq C_{q,\nu} \sqrt{n \varepsilon} \left( \left( \frac{4R}{r}\right)^{\nu} + \left( \frac{r}{\rho} \right)^{\nu \left(2\nu - 1 \right)} \right) + C \sqrt{ \delta} r
\end{split}
\end{equation}
for $4R \leq r \leq \rho / 2.$ Assuming without loss that $E(u(\tau)) \geq \eps,$ (\ref{lowdbarmaintechnical:assumption1}) gives
\begin{equation*}
\rho \leq \frac{\sqrt{\eps}}{C\sqrt{\delta}}
\end{equation*}
and
\begin{equation*}
\sqrt{\delta} < \frac{\sqrt{\eps} }{\rho}.
\end{equation*}
So (\ref{lowdbarmaintechnical:4}) simplifies to
\begin{equation}\label{lowdbarmaintechnical:5}
\begin{split}
r|du(t)| & \leq C_{q,\nu} \sqrt{n \varepsilon} \left( \left( \frac{R}{r}\right)^{\nu} + \left( \frac{r}{\rho} \right)^{\nu \left(2\nu - 1 \right) } \right).
\end{split}
\end{equation}
The rest of the argument proceeds as before, using (\ref{lowdbarmaintechnical:5}) in place of (\ref{maintechnical:5}).
\end{proof}

\begin{cor}\label{cor:nojointconcentration}
Assume that $N$ has nonnegative holomorphic bisectional curvature, and that for every $p \in \Sigma,$ either
\begin{equation}\label{nojoinconcentration:assumption}
\lim_{\rho \searrow 0} \limsup_{t \nearrow T} E_{\bp} (u(t), B_\rho (p)) < \delta_0 \,\,\, \text{   or   } \,\,\, \lim_{\rho \searrow 0}  \limsup_{t \nearrow T} E_\p(u(t), B_\rho(p)) < \delta_0.
\end{equation}
Then $u(T)$ is $C^\mu$ for each $\mu < 1.$ 

\end{cor}
\begin{proof} Given $\mu < 1,$ choose $\nu < 1$ so that $\nu (2 \nu - 1) = \mu,$ and let $\eps = \eps_3 / 2.$
Fix $p \in \Sigma,$ and suppose without loss of generality that
$$\lim_{\rho \searrow 0} \limsup_{t \nearrow T} E_{\bp} (u(t), B_\rho (p)) < \delta_0.$$
Then for $\rho_1 > 0$ sufficiently small, (\ref{lowdbarmaintechnical:assumption1}) holds for $\tau$ sufficiently close to $T.$
The rest of the proof proceeds as in Corollary \ref{cor:stressboundedholdercont}.
\end{proof}

\begin{rmk} By the standard bubble-tree decomposition and energy identity at a finite-time singularity \cite{Qing1995, DingTian1995, Wang1996}, (\ref{nojoinconcentration:assumption}) is equivalent to the assumption that only holomorphic or antiholomorphic bubbles appear at each singular point. This hypothesis is familiar from Topping's second convergence theorem \cite{Topping2004-annmath}.
\end{rmk}

\vspace{5mm}

\subsection{Bubble-tree convergence} We can now state the bubble-tree results that follow from our main theorems. The arguments are by now standard (see Ding-Tian \cite{DingTian1995}, Parker \cite{Parker1996}, or Qing-Tian \cite{QingTian1997}), once one establishes the following key estimates on the energy and oscillation in the neck region (cf. Lin-Wang \cite{LinWang1998}, Lemma 3.1).

\begin{lemma}\label{lemma:noneck} Let $0 < \rho_i < R_0,$ $i = 1, \ldots, \infty,$ be any sequence of numbers and $p_i \in \Sigma$ any sequence of points. 
Let $u_i : \Sigma \times \LB 0, T \right) \to N$ be a sequence of solutions of harmonic map flow with uniformly bounded energy. Suppose that there exists $0 < \tau < T$ such that a uniform stress-energy bound of the form (\ref{stressbound}) holds, with $q > 1,$ as well as
\begin{equation}\label{noneck:delta0assn}
\int_{\tau}^T \!\!\!\!\! \int_\Sigma |\T(u_i) |^2\, dV dt < \frac{c \varepsilon_2^2}{E(u_i(\tau))}
\end{equation}
for each $i,$ where $\varepsilon_2 > 0$ is the constant of Corollary \ref{cor:delta0}. 
Given any sequence of times $t_i \nearrow T,$ let $\lambda_i = \lambda_{\varepsilon_1,\rho_i, p_i}(u(t_i)),$ and assume that $\lambda_i / \rho_i \to 0 \mbox{ as } i\to \infty.$
Then we have
\begin{equation}\label{noneck:Eest}
\lim_{\beta \searrow 0} \lim_{\alpha \to \infty} \lim_{i \to \infty} E( u_i(t_i), U_{\alpha \lambda_i}^{\beta \rho_i}(p_i) ) = 0
\end{equation}
and
\begin{equation}\label{noneck:Oscest}
\lim_{\beta \searrow 0} \lim_{\alpha \to \infty} \lim_{i \to \infty} \mbox{osc}_{U_{\alpha \lambda_i}^{\beta \rho_i}(p_i) } u_i(t_i) = 0.
\end{equation}
\end{lemma}
\begin{proof} By Corollary \ref{cor:delta0}, we have
\begin{equation}\label{noneck:copiedest}
r| du_i(x,t_i) | \leq C \left( \frac{\lambda_i}{r} + \frac{r}{\rho_i} \right)^{1 - \frac{1}{q}}
\end{equation}
for $ 2\lambda_i \leq |x| \leq \rho_i / 2.$ 
Then (\ref{noneck:Eest}-\ref{noneck:Oscest}) follow by integrating (\ref{noneck:copiedest}).
\end{proof}

\begin{thm}\label{thm:bubbletree} Let $u_i : \Sigma \times \LB 0, T \right) \to N,$ where $0 < T \leq \infty,$ be a sequence of solutions of harmonic map flow with uniformly bounded energy, and suppose that there exists $0 < \tau < T$ such that (\ref{stressbound}) holds, with $q > 1$ and a uniform constant. 
Suppose further that
\begin{equation}\label{bubbletree:tensionassn}
\lim_{t \nearrow T} \limsup_{i \to \infty} \int_{t}^T \!\!\!\!\! \int_\Sigma |\T(u_i) |^2\, dV ds = 0.
\end{equation}
Given any sequence of times $t_i \nearrow T,$ there exists a subsequence (again labeled $i$) such that $u_i(t_i)$ converges to a $C^{1 - \frac{1}{q}}$ weak limit $u_\infty: \Sigma \to N,$ smoothly away from a finite set 
$S \subset \Sigma.$ For each $p \in S,$ there exist finitely many harmonic maps $\phi_k : S^2 \to N$ such that the following energy identity holds:
\begin{equation}\label{energyidentity}
\lim_{r \searrow 0} \lim_{i \to \infty} E(u_i(t_i), B_r(p)) = \sum_k E(\phi_k).
\end{equation}
Moreover, there are no necks, {\it i.e.}, $\cup_k \phi_k(S^2)$ is connected and contains $\lim_{x \to p} u_\infty(x).$ 
\end{thm}
\begin{proof} (\emph{Sketch}) 
It is most natural to use the ``outside-in'' argument of Parker \cite{Parker1996}.

In view of the parabolic estimates of Theorem \ref{thm:epsilonreg}, after passing to a subsequence, we may assume that $u_i(t_i)$ converges smoothly to $u_\infty$ 
away from a finite set $S$, and the limit
\begin{equation}\label{E0defn}
E_0 = \lim_{r \searrow 0} \lim_{i \to \infty} E(u_i(t_i), B_r (p)) \geq \varepsilon_0
\end{equation}
exists at each $p \in S.$ In view of (\ref{bubbletree:tensionassn}), after advancing $\tau,$ we can assume that (\ref{noneck:delta0assn}) holds. By the argument of Lemma \ref{lemma:lambdazero}, after again passing to a subsequence, we may choose $\rho > 0$ such that $\lambda_i^\circ = \lambda_{\varepsilon_2, \rho, p}(u_i(t_i)) \to 0.$ Since (\ref{noneck:delta0assn}) implies (\ref{delta0:assumption}), we may conclude as in the proof of Corollary \ref{cor:stressboundedholdercont} that $u_\infty$ is $C^{1 - \frac{1}{q}}.$ 

Now, in geodesic coordinates where $p = 0,$ we may let $q_i$ be the ``center-of-mass'' of $e(u_i(t_i))$ over the ball $B_{\lambda^\circ_i}(0),$ and put
\begin{equation*}
\lambda_i = \lambda_{\varepsilon_2, \rho/2, q_i}(u_i(t_i)).
\end{equation*}
It follows from Corollary \ref{cor:delta0} and (\ref{E0defn}) that $\lambda_i> 0$ for each $i,$ and $\lambda_i \leq C \lambda_i^\circ.$

Consider the rescaled sequence
\begin{equation}
v_i(x) = u_i(q_i + \lambda_i x, t_i).
\end{equation}
Then $v_i(x)$ is defined for $x \in B_{\rho/2 \lambda_i }(0) \subset \R^2,$ and satisfies
\begin{equation}\label{lambdaequalsone}
\lambda_{\varepsilon_1,\rho/2 \lambda_i , 0}(v_i) = 1.
\end{equation}
The sequence $v_i$ approaches a weak limit $\phi_1 : \R^2\setminus \{p'_1, \cdots, p'_n\} \to N,$ where $p_i' \in \bar{B}_1$ by (\ref{lambdaequalsone}). Since $\lambda_i \to 0,$ 
the assumption (\ref{bubbletree:tensionassn}) and the estimate (\ref{epsilonreg:tensionfield}) imply that $\phi_1$ is a harmonic map. Hence, by $\eps$-regularity (applied to $\phi_1$ on large balls in $\R^2$), either $E(\phi_1, \R^2) \geq \varepsilon_0$ or $\phi_1$ is constant.

We claim that if $\phi_1$ is constant, then there must be at least two distinct points in the singular set. Assume for contradiction that there is only one such point $p'_1.$ Then since $\phi_1$ is constant, all of the energy must concentrate at $p'_1$ as $i \to \infty;$ since the origin is the center-of-mass, we must have $p'_1 = 0.$ But then we clearly have $E(v_i, U_{1/4}^1(0)) < \varepsilon_1$ for $i$ sufficiently large, which implies that $\lambda_{\varepsilon_1, \rho/2 \lambda_i , 0}(v_i) < 1.$ This contradicts (\ref{lambdaequalsone}).

We are left with two scenarios: either $\phi_1$ is nonconstant, with energy at least $\varepsilon_0,$ or there are at least two distinct points in the bubbling set of $\{v_i\},$ each consuming energy at least $\varepsilon_0.$ By Lemma \ref{lemma:noneck}, we have
\begin{equation}\label{protoenergy}
E_0 = E(\phi_1) + \sum_j \lim_{r \searrow 0} \lim_{i \to \infty} E(v_i, B_r(p'_j) )
 \end{equation}
 and
 \begin{equation}\label{limutiphi}
 \lim_{q \to p} u(t_i, q) = \lim_{x \to \infty} \phi_1(x).
 \end{equation}
In either case, for each $p'_j,$ we must have
\begin{equation}\label{energygoesdown}
\lim_{r \searrow 0} \lim_{i \to \infty} E(v_i, B_r(p'_j) ) < E_0 - \varepsilon_0.
\end{equation}
We can now rescale around each point $p'_j$ in the same fashion, and repeat the procedure. In view of (\ref{energygoesdown}), the amount of concentrated energy at each point goes down by at least $\varepsilon_0,$ therefore the process must terminate in finitely many steps. The identity (\ref{protoenergy}) yields the energy identity \ref{energyidentity}), and (\ref{limutiphi}) yields the no-neck property.
\end{proof}

\begin{cor}\label{cor:noneck} Let $u: \Sigma \times \LB 0, T \right) \to N,$ where $T \leq \infty,$ be a solution of harmonic map flow which is smooth for $0 < t < T.$ Assume that a uniform stress-energy bound (\ref{stressbound}) holds, with $q > 1.$ 
Given any sequence of times $t_i \nearrow T,$ the maps $u(t_i)$ converge along a subsequence in the bubble-tree sense with no necks, where the body map is H\"older continuous and the bubble maps are harmonic. 
If $T = \infty,$ the body map is harmonic. 
\end{cor}
\begin{proof} This follows by letting $u_i = u$ in the previous theorem, 
in which case (\ref{bubbletree:tensionassn}) follows from the global energy identity (\ref{globalenergyidentity}). 
The fact that $u_\infty$ is harmonic for $T = \infty$ (without any assumptions) follows from the global energy identity and (\ref{epsilonreg:tensionfield}) by a standard argument, which we omit. 
\end{proof}

\vspace{5mm}

\subsection{Proof of Theorem \ref{thm:main}} Let $u : \Sigma \times \LB 0, \infty \right) \to N$ be a Struwe solution of (\ref{e:hm}). As observed above, (\ref{globalenergyidentity}) and (\ref{Esplitting}) imply that $E_{\bp}(u(t))$ is decreasing along the flow; this remains true through finite-time singularities, since $u(T)$ is a strong limit away from a set of measure zero. 
Hence, we have
\begin{equation}\label{Edbarsmallt}
E_\dbar(u(t)) < \delta_0
\end{equation}
for all $0 \leq t < \infty.$ H{\"o}lder continuity of the body maps now follows from Corollary \ref{cor:nojointconcentration}. By Corollary \ref{cor:nnhbscblowuprate}, (\ref{Edbarsmallt}) implies (\ref{stressbound}) with $q = 2,$ hence the bubble-tree and no-neck properties follow from Corollary \ref{cor:noneck}.

\end{document}